\documentclass[12pt, A4paper]{amsart}
  \usepackage{latexsym}
  \usepackage[all]{xy}
  \usepackage{amsthm}
  \usepackage{amssymb}
 \numberwithin{equation}{section}
  \usepackage{color}
  \usepackage{hyperref}

\def\SM{\mathsf{SetMat}}
\def\Cat{\mathsf{Cat}}
\def\cE{\mathcal E}
\def\cM{\mathcal M}
\def\cC{\mathcal C}
\def\cO{\mathcal O}
\def\cD{\mathcal D}
\def\ad2{(-)^2}
\def\x{\times}
\def\lt{\triangleleft}
\def\rt{\triangleright}
\def\emor{\twoheadrightarrow}
\def\mmor{\rightarrowtail}

  \newtheorem{proposition}{Proposition}[section]
  \newtheorem*{proposition*}{Proposition}
  \newtheorem*{lemma*}{Lemma}
  \newtheorem{lemma}[proposition]{Lemma}
  \newtheorem{corollary}[proposition]{Corollary}
  \newtheorem{theorem}[proposition]{Theorem}

  \theoremstyle{definition}
  
  \newtheorem{example}[proposition]{Example}

  \newtheorem{aussage}[proposition]{}

  \newcounter{c}
  
  \newcommand{\etyk}[1]{\vspace{-7.4mm}$$\begin{equation}\Label{#1}
  \addtocounter{c}{1}}
  \renewcommand{\]}{\ifnum \value{c}=1 $$\else \end{equation}\fi}
\begin{document}

\title{Factorization systems induced by weak distributive laws}

 \author{Gabriella B\"ohm}
 \address{Research Institute for Particle and Nuclear Physics, Budapest,
 \newline\indent H-1525
 Budapest 114, P.O.B.\ 49, Hungary}
  \email{G.Bohm@rmki.kfki.hu}
\date{Sep 2010}

\begin{abstract}
We relate weak distributive laws in $\SM$ to strictly  associative (but not
strictly unital) pseudoalgebras of the 2-monad $\ad2$ on $\Cat$. The
corresponding orthogonal factorization systems are characterized by a certain
bilinearity property.  
\end{abstract}
  \maketitle

%%%%%%%%%%%%%%%%%%%%%%%%%%%   INTRODUCTION %%%%%%%%%%%%%%%%%%%%%%%%%%%%%%%%%%

\section*{Introduction}

The study of (orthogonal) factorization systems has a long history in category 
theory, see e.g. \cite{FreKel}. The interest in the subject was renewed by
its appearance in Quillen's model categories.

An {\em orthogonal factorization system} (see e.g. \cite{FreKel}) on a small
category $\cC$ is given by subcategories $\cE\hookrightarrow \cC\hookleftarrow
\cM$, such that both $\cE$ and $\cM$ contain all isomorphisms of $\cC$; any
morphism $f$ in $\cC$ admits a (not necessarily unique) decomposition $f=me$
in terms of morphisms $e$ in $\cE$ and $m$ in $\cM$; and the so-called
`diagonal fill-in' condition holds, i.e. for any morphisms $e$ in $\cE$ and
$m$ in $\cM$, every commutative diagram  
$$
\xymatrix{x\ar[r]^-u\ar[d]_-e&
y\ar[d]^-{m}\\
a\ar[r]^-{v}\ar@{-->}[ru]^-t&
b}
$$
in $\cC$ admits a uniquely determined fill-in morphism $t$ in $\cC$ such that
both resulting triangles commute. As it was observed by Korostenski and Tholen 
in \cite{KorTho}, orthogonal factorization systems can be described as
pseudoalgebras of the 2-monad $\ad2$ on the 2-category $\Cat$ of categories;
functors; and natural transformations. 

A {\em strict factorization system} (see \cite{Gra}) on $\cC$ means
subcategories $\cE\hookrightarrow \cC\hookleftarrow \cM$, such that both $\cE$
and $\cM$ contain all identity morphisms; and any morphism $f$ in $\cC$ admits
a {\em unique} decomposition $f=me$ in terms of morphisms  $e$ in $\cE$ and
$m$ in $\cM$. As it was pointed out by Grandis in \cite{Gra}, strict
factorization systems can be regarded as special instances of orthogonal ones
-- in the sense that any strict factorization system is contained in precisely
one orthogonal factorization system. It was proved by Rosebrugh and Wood in
\cite{RosWoo} that orthogonal factorization systems arising from strict ones
are precisely those corresponding to strict algebras for the 2-monad $\ad2$ on
$\Cat$ via the correspondence in \cite{KorTho}. Furthermore, in \cite{RosWoo}
also an equivalence between strict factorization systems and distributive laws
(in the sense of \cite{Beck} or rather of \cite{street}) in the bicategory
$\SM$ was established.     

Our aim in this paper is to study factorization systems induced by {\em weak
distributive laws} (in the sense introduced by Ross Street in
\cite{Str:weak_dl}) in $\SM$.  While a distributive law in $\SM$ determines a
strict factorization system; i.e. an inverse of the 2-cell 
$$
\xymatrix{
\cE\cM\ar@{^{(}->}[r]&
\cC\cC\ar[r]^-{\circ}& 
\cC}
$$
in $\SM$, a weak distributive law corresponds to an appropriate section of
this 2-cell, obeying `$\cE$-, and $\cM$-linearity' properties (explained in
detail in Section \ref{sec:demifact}).  
We call such a section a {\em bilinear factorization system}.  
We prove that there are adjunctions whose counits are isomorphisms, 
between the following categories. 
First, the category of weak distributive laws in $\SM$, to be defined in
Section \ref{sec:wdl}; 
and the category of bilinear factorization systems on small categories,
introduced in Section \ref{sec:demifact}.
Second, this category of bilinear factorization systems; and the category of
strictly associative pseudoalgebras of the 2-monad $\ad2$ on $\Cat$.
Since by results in \cite{KorTho},  pseudoalgebras of $\ad2$ are equivalent to
orthogonal factorization systems, we conclude that any bilinear factorization
system is contained in a canonically associated orthogonal factorization
system. In particular, weak distributive laws in $\SM$ induce orthogonal
factorization systems. 
\bigskip

\noindent
{\bf Acknowledgement.}
This work was supported by the Hungarian Scientific Research Fund OTKA
F67910. 
My interest in the subject of this paper was triggered by Emily Riehl's talk
at The Third Morgan-Phoa Mathematics Workshop in Canberra, January of 2010. I
am grateful to Emily for an inspiring talk and the organizers for a generous
invitation. 

%%%%%%%%%%%%%%%%%%%%%%%%%%%%%%%%%%%  SEC 1 %%%%%%%%%%%%%%%%%%%%%%%%%%%%%%%%%%

\section{Weak distributive laws in $\SM$}
\label{sec:wdl}

\begin{aussage}{\bf Weak distributive laws.} \label{as:wdl}
Extending the notion of distributive law due to Jon Beck (see \cite{Beck}),
weak distributive laws in a bicategory  were introduced by Ross Street in
\cite{Str:weak_dl} as follows.  
They consist of two monads $(t,\mu,\eta)$, $(s, \nu,\vartheta)$ on the same
object, and a 2-cell $\varphi:ts\to st$ such that the following diagrams
commute (where the coherence 2-cells are omitted, as they are throughout the
paper).   
$$
\xymatrix@C=10pt{
tts\ar[r]^-{t\varphi}\ar[d]^-{\mu s}&
tst\ar[r]^-{\varphi t}&
stt\ar[d]_-{s\mu}\\
ts\ar[rr]_--\varphi&&
st
}\ 
\xymatrix@C=10pt{
s\ar[rrr]^-{\eta s}\ar[d]^-{s \eta}&&&
ts\ar[d]_-\varphi\\
st\ar[r]_-{st \vartheta}&
sts\ar[r]_-{s\varphi}&
sst\ar[r]_-{\nu t}&
st
}\ 
\xymatrix@C=10pt{
tss\ar[r]^-{\varphi s}\ar[d]^-{t\nu}&
sts\ar[r]^-{s\varphi}&
sst\ar[d]_-{\nu t}\\
ts\ar[rr]_--\varphi&&
st
}\ 
\xymatrix@C=10pt{
t\ar[rrr]^-{t\vartheta}\ar[d]^-{\vartheta t}&&&
ts\ar[d]_-\varphi\\
st\ar[r]_-{\eta st}&
tst\ar[r]_-{\varphi t}&
stt\ar[r]_-{s\mu}&
st
}
$$
By \cite[Proposition 2.2]{Str:weak_dl}, the second and fourth diagrams can
be replaced by a single diagram
$$
\xymatrix{
st\ar[r]^-{\eta st}\ar[d]_-{st\vartheta}&
tst\ar[r]^-{\varphi t}&
stt\ar[d]^-{s\mu}\\
sts\ar[r]_-{s\varphi}&
sst\ar[r]_-{\nu t}&
st\ .
}
$$
The equal paths in the latter diagram give rise to an idempotent 2-cell which
is an identity if and only if $\varphi$ is a distributive law in the strict
sense. 
\end{aussage}

\begin{aussage}{\bf Demimonads.} \label{as:demimonad}
In simplest terms, a demimonad in a bicategory is a monad in the local Cauchy
completion, cf. \cite{Lac_talk}. Explicitly, it is given by a 1-cell $t:A\to
A$ and 2-cells $\mu:t^2 \to t$ and $\eta:1_A\to t$ such that the following
diagrams commute. 
$$
\xymatrix{
t^3\ar[r]^-{\mu t}\ar[d]_-{t\mu}&
t^2\ar[d]^-\mu\\
t^2\ar[r]_-\mu&
t
}
\quad 
\xymatrix{
t\ar[r]^-{\eta t}\ar[d]_-{t\eta}&
t^2\ar[d]^-\mu\\
t^2\ar[r]_-\mu&
t
}\quad
\xymatrix{
1_A \ar[r]^-{\eta \eta}\ar[dr]_-{\eta}&
t^2\ar[d]^-\mu\\
&t
}\quad
\xymatrix{
t^2\ar[r]^-{\eta t^2}\ar[rd]_-{\mu}&
t^3\ar[d]^-{\mu^2}\\
&t
} 
$$
This structure occurred in \cite{Boh_wtm} under the name `pre-monad'. The
2-cell $\mu.\eta t: t\to t$ is idempotent and whenever it splits, the
corresponding retract is a proper monad.

By a special case of \cite[Theorem 2.3]{Boh_wtm}, for any weak distributive
law $\varphi:ts\to st$, the composite 1-cell $st$ is a demimonad via
$\varphi.\eta\theta:1_A\to st$ and $\nu\mu.s\varphi t:stst\to st$. 
\end{aussage}

\begin{aussage}{\bf The bicategory $\SM$.} \label{as:SetMat}
Recall (e.g. from Section 2 of \cite{DayStr}) that
0-cells in $\SM$ are sets. 1-cells $f:A\to B$ are set valued matrices; that
is, collections $\{f(a,b)\}_{a,b}$ of sets labelled by elements $a$ of $A$ and
$b$ of $B$. 2-cells $\alpha:f\to g$ are matrices $\{\alpha(a,b)\}_{a\in A,b\in
B}$ of functions $\alpha(a,b):f(a,b)\to g(a,b)$. Horizontal composition is
given by matrix multiplication; that is, $hf(a,c)=\sum_{b\in B} h(b,c)\x
f(a,b)$, for 1-cells $f:A\to B$ and $h:B\to C$. Horizontal composition of
2-cells is given by the obvious inducement. Vertical composition is given by
the elementwise composition of functions.  
  
Functions $q:A \to B$ can be regarded as 1-cells $f_q:A\to B$ by taking
$f_q(a,b)$ to be a one element set if $b=q(a)$ and the empty set otherwise.  

Small categories are the same as monads in $\SM$. Functors $\cM\to
\cM'$ between them can be described as appropriate monad morphisms (in the
sense of \cite{street}) -- whose 1-cell part $f_q$ is induced by the
object map $q$ and whose 2-cell part $f_q\cM\to \cM'f_q$ is given by the
morphism map $\cM(a,a')\to \cM'(q(a),q(a'))$. 
\end{aussage}

\begin{aussage}{\bf Weak distributive laws in $\SM$.}
\label{as:wdl_SM}
Applying the definition of a weak distributive law in \ref{as:wdl} to the
particular bicategory $\SM$, we arrive at the following structure. We need two
categories $\cE$ and $\cM$ with coinciding object sets $\cO$, and for all
elements $a,b,c$ of $\cO$ a function
$$
\cE(b,c)\x \cM(a,b) \to \sum_{x\in \cO} \cM(x,c)\x \cE(a,x),\qquad 
(e,m)\mapsto (e\rt m, e\lt m),
$$
where $e\lt m$ denotes a morphism $a\to e.m$ in $\cE$, $e\rt m$ is a morphism
$e.m\to c$ in $\cM$ and $e.m$ is an element of $\cO$. The four conditions in
\ref{as:wdl} translate to the following requirements, for all morphisms
$e:c\to b'$ and $f:b'\to a'$ in $\cE$ and $n:a\to b$ and $m:b\to c$ in $\cM$.
\begin{eqnarray}
1_c.m=1_b.1_b\qquad &&e.1_c=1_{b'}.1_{b'} \label{eq:dot_1}\\
(fe).m=f.(e\rt m)\qquad &&e.(mn)=(e\lt m).n \label{eq:dot_mp}\\
1_c\lt m = 1_b \lt 1_b\qquad&& e\lt 1_c=(1_{b'} \lt 1_{b'})e \label{eq:lt_1}\\
(fe)\lt m=[f\lt(e\rt m)](e\lt m) \qquad && e\lt(mn)=(e\lt m)\lt n 
\label{eq:lt_mp}\\
1_c\rt m = m(1_b\rt 1_b)\qquad &&e\rt 1_c = 1_{b'} \rt 1_{b'} \label{eq:rt_1}\\
(fe)\rt m =f\rt(e\rt m)\qquad &&e\rt (mn)=(e\rt m)[(e\lt m)\rt n].
\label{eq:rt_mp}
\end{eqnarray}
\end{aussage}

\begin{aussage} {\bf The category induced by a weak distributive law in
$\SM$.}  \label{as:prod_cat}
As we recalled in \ref{as:demimonad}, for a weak distributive law
$\varrho:\cE\cM\to \cM\cE$ in $\SM$, $\cM\cE$ is a demimonad in $\SM$ via the
multiplication $(n,f)(m,e)=(n(f\rt m),(f\lt m)e)$ and $\eta(a)=(1_a\rt
1_a,1_a\lt 1_a)$. The corresponding idempotent 2-cell $\cM\cE\to \cM\cE$ in
$\SM$ takes the explicit form  
\begin{equation}\label{eq:ErhoM_morphisms}
(m,e)\mapsto (1_c\rt m, e\lt 1_a)=(m(1_b\rt 1_b),(1_b\lt 1_b)e),
\end{equation}
for any morphisms $e:a\to b$ in $\cE$ and $m:b\to c$ in $\cM$. Since this is 
clearly split, the corresponding retract is a monad in $\SM$, that is, a small
category. It will be denoted by $\cM_\varrho\cE$. Its objects are the same as
the objects of $\cE$ and of $\cM$, and its morphisms have the form as on the
right hand side of \eqref{eq:ErhoM_morphisms}. By the axioms of a demimonad in
\ref{as:demimonad}, the identity morphism of any object $a$ is equal to
$\eta(a)=(1_a\rt 1_a,1_a\lt 1_a)$ and the composition law is
\begin{equation}\label{eq:comp_ErhoM}
(1_c\rt n,f\lt 1_b)(1_b\rt m,e\lt 1_a)=(n(f\rt m),(f\lt m)e),
\end{equation} 
for morphisms $e:a\to x$ and $f:b\to y$ in $\cE$ and $m:x\to b$ and $n:y\to c$
in $\cM$. (Note that $1_c \rt (n(f\rt m))=n(f\rt m)$ and $((f\lt m)e)\lt
1_a=(f\lt m)e$.) 

Since we are dealing simultaneously with three categories $\cE$, $\cM$ and
$\cM_\varrho\cE$, the following arrow notation (which we learnt from 
\cite{RosWoo}) will turn out to be handy. For morphisms in $\cE$, we use
arrows of the type $\emor$, for morphisms in $\cM$ we draw $\mmor$. Arrows in
$\cM_\varrho \cE$ can be drawn as
$$
\xymatrix{
a\ar@{->>}[rr]^-{e\lt 1_a} &&
1_b.1_b\   \ar@{>->}[rr]^-{1_c\rt m}&&
c.
}
$$
\end{aussage}

\begin{proposition}\label{prop:functor_i}
For a weak distributive law $\varrho:\cE\cM\to \cM\cE$ in $\SM$, the following 
functions define identity-on-objects functors $i:\cE\to \cM_\varrho\cE$ and
$i:\cM\to \cM_\varrho\cE$, respectively.
\begin{eqnarray*}
&(\xymatrix{a\ar@{->>}[r]^-e& b}) \mapsto
&(\xymatrix{a\ar@{->>}[r]^-{e\lt 1_a}&1_b.1_b\  \ar@{>->}[r]^-{1_b\rt 1_b}& b})
\quad \textrm{and}\\
&(\xymatrix{a\  \ar@{>->}[r]^-m&b}) \mapsto
&(\xymatrix{a\ar@{->>}[r]^-{1_a\lt 1_a}&1_a.1_a\  \ar@{>->}[r]^-{1_b\rt m}& b}).
\end{eqnarray*}
\end{proposition}

\begin{proof}
Preservation of identity morphisms is obvious. In order to verify preservation
of compositions, note that $i(e)$ is the image of $(1_b,e)$, and $i(m)$ is the
image of $(m,1_a)$, under the idempotent 2-cell $\mu.t\eta$ in the demimonad
$t=\cM\cE$. Hence preservation of compositions follows from the demimonad
identity $\mu. \mu\mu.t \eta t \eta = \mu$.
When the functor $\cE\to \cM_\varrho\cE$ is considered, evaluate this
demimonad identity at any object of the form $(1_c,f,1_b,e)$ of
$\cM\cE\cM\cE(a,c)$; and use that by the second equalities in \eqref{eq:lt_1}
and \eqref{eq:rt_1}, $(1_c,f)(1_b,e)=i(fe)$. One proceeds symmetrically in
case of the functor $\cM\to \cM_\varrho\cE$.  
\end{proof}

\begin{lemma}\label{lem:rho_commutes}
For a weak distributive law $\varrho:\cE\cM\to \cM\cE$ in $\SM$, and for
compatible morphisms $e$ in $\cE$ and $m$ in $\cM$, the following diagram in
$\cM_\varrho\cE$ is commutative,
$$
\xymatrix{
a\ar[r]^{i(m)}\ar[d]_-{i(e\lt m)}&b\ar[d]^-{i(e)}\\
e.m\ar[r]_-{i(e\rt m)}&c
}
$$
where $i$ denotes both functors in Proposition \ref{prop:functor_i}.
\end{lemma}

\begin{proof}
Apply the demimonad identity $\mu. \mu\mu.t \eta t \eta = \mu$ to the
demimonad $t=\cM\cE$ and evaluate it on any element of the form
$(n,1_b,1_b,f)$ of $\cM\cE\cM\cE(a,c)$. It yields
\begin{equation}\label{eq:inif}
i(n)i(f)=\big(\xymatrix{
a\ar@{->>}[r]^-{f\lt 1_a}&1_b.1_b\, \ar@{>->}[r]^-{1_c\rt n}&c
}\big).
\end{equation}
Hence the down-then-right path is equal to the morphism $(e\rt m, e\lt m)$ in
$\cM_\varrho\cE$. Evaluating the same demimonad identity now on the element 
$(1_c,e,m,1_a)$ of $\cM\cE\cM\cE(a,c)$, we conclude that also the
right-then-down path is equal to the morphism $(e\rt m,e\lt m)$ in
$\cM_\varrho\cE$. 
\end{proof}

\begin{aussage}{\bf Injective weak distributive laws in $\SM$.}
We say that a weak distributive law $\varrho:\cE\cM\to \cM\cE$ in $\SM$ is
{\em injective} provided that both functors $\cE\to \cM_\varrho\cE$ and
$\cM\to \cM_\varrho\cE$ in Proposition \ref{prop:functor_i} are inclusions,
i.e. they act as injective functions on the morphisms. Equivalently, a weak
distributive law $\varrho:\cE\cM\to \cM\cE$ in $\SM$ is injective if an only
if $1_a\lt 1_a$ is a monomorphism in $\cE$ and $1_a\rt 1_a$ is an epimorphism
in $\cM$, for all objects $a$.
\end{aussage}

\begin{aussage}{\bf Morphisms between weak distributive laws in $\SM$.}
A morphism $(\varrho:\cE\cM\to \cM\cE)\to (\varrho':\cE'\cM'\to \cM'\cE')$ of
weak distributive laws in $\SM$ is by definition a pair of functors
$Q_\cE:\cE\to \cE'$ and $Q_\cM:\cM\to \cM'$ with a common object map $q$, such
that the following diagram (of 2-cells in $\SM$) commutes,
$$
\xymatrix{
f_q\cE\cM\ar[d]_-{f_q\varrho} \ar[r]^-{Q_\cE \cM}&
\cE'f_q\cM \ar[r]^-{\cE' Q_\cM}&
\cE'\cM'f_q\ar[d]^-{\varrho' f_q}\\
f_q \cM\cE  \ar[r]^-{Q_\cM \cE}&
\cM'f_q\cE \ar[r]^-{\cM' Q_\cE}&
\cM'\cE'f_q
}
$$
where $f_q$ is the 1-cell induced by $q$, cf. \ref{as:SetMat}.
In other words, commutativity of this diagram means the following equalities,
for all morphisms $m:a\mmor b$ in $\cM$ and $e:b\emor c$ in $\cE$.
$$
q(e.m)=Q_\cE(e).Q_\cM(m),\ \ 
Q_\cE(e\lt m)=Q_\cE(e)\lt Q_\cM(m),\ \ 
Q_\cM(e\rt m)=Q_\cE(e)\rt Q_\cM(m).
$$
Weak distributive laws in $\SM$ and their morphisms constitute a category.
\end{aussage}

\begin{proposition} \label{prop:retract}
There is an adjunction whose counit is an isomorphism, between the category of
all weak distributive laws, and the category of injective weak distributive
laws in $\SM$. 
\end{proposition}

\begin{proof}
A to-be-left-adjoint $S$ of the inclusion functor $J$ is constructed as follows.
For a weak distributive law $\varrho:\cE\cM\to \cM\cE$, by Proposition
\ref{prop:functor_i} there are identity-on-objects functors $i$ from $\cE$ and 
from $\cM$ to the associated product category $\cM_\varrho\cE$. Hence we may
consider  the subcategories $i(\cE)$ and $i(\cM)$ of $\cM_\varrho\cE$. An
(evidently injective) weak distributive law $S(\varrho:\cE \cM\to
\cM\cE):=({\widetilde \varrho}:i(\cE)i(\cM)\to i(\cM)i(\cE))$ is given by
\begin{eqnarray*}
{\widetilde \varrho}(i(e),i(m))&&\hspace{-.6cm}=
{\widetilde \varrho}(
\xymatrix{b\ar@{->>}[r]^-{e\lt 1_b}&1_c.1_c\  \ar@{>->}[r]^-{1_c\rt 1_c}&c},
\xymatrix{a\ar@{->>}[r]^-{1_a\lt 1_a}&1_a.1_a\  \ar@{>->}[r]^-{1_b\rt m}&b}
):=\\
&&\hspace{-.3cm}
(\xymatrix{e.m\ar@{->>}[r]^-{1_{e.m}\lt 1_{e.m}}&e.m\  \ar@{>->}[r]^-{e\rt m}&c},
\xymatrix{a\ar@{->>}[r]^-{e\lt m}&e.m\  \ar@{>->}[r]^-{1_{e.m}\rt
    1_{e.m}}&e.m})
=(i(e\rt m),i(e\lt m)).
\end{eqnarray*}
Observe that the pair $(i(e\rt m),i(e\lt m))$ is uniquely determined by the
composite morphism
$$
i(e\rt m)i(e\lt m)=(e\rt m,e\lt m)=i(e)i(m)
$$
in $\cM_\varrho\cE$, cf. Lemma \ref{lem:rho_commutes}. Thus the 2-cell
$\widetilde \varrho$ in $\SM$ is well-defined. Since $i$ stands for functors
and $\varrho$ was a weak distributive law in $\SM$, so is $\widetilde
\varrho$ and $(i:\cE\to i(\cE),i:\cM\to i(\cM))$ is a morphism of weak
distributive laws.

A morphism $(Q_\cE:\cE\to \cE'$, $Q_\cM:\cM\to \cM')$ of weak distributive
laws induces functors ${\widetilde Q}_\cE:i(\cE)\to i'(\cE')$ and ${\widetilde 
Q}_\cM:i(\cM)\to i'(\cM')$,
\begin{eqnarray*}
{\widetilde Q}_\cE(i(e))&=&
{\widetilde Q}_\cE(
\xymatrix{
a\ar@{->>}[r]^-{e\lt 1_a}&1_b.1_b\  \ar@{>->}[r]^-{1_b\rt 1_b}&b}):=\\
&&(\xymatrix@=40pt{
q(a)\ar@{->>}[r]^-{Q_\cE(e\lt 1_a)}&q(1_b.1_b)\  
\ar@{>->}[r]^-{Q_\cM(1_b\rt 1_b)}&q(b)})=
i'(Q_\cE(e))\\
{\widetilde Q}_\cM(i(m))&=&
{\widetilde Q}_\cM(
\xymatrix{
a\ar@{->>}[r]^-{1_a\lt 1_a}&1_a.1_a\  \ar@{>->}[r]^-{1_b\rt m}&b}):=\\
&&(\xymatrix@=40pt{
q(a)\ar@{->>}[r]^-{Q_\cE(1_a\lt 1_a)}&q(1_a.1_a)\  
\ar@{>->}[r]^-{Q_\cM(1_b\rt m)}&q(b)})=
i'(Q_\cM(m)),
\end{eqnarray*}
where $q$ denotes the common object map of the functors $Q_\cE$ and
$Q_\cM$. The last ones of the above equalities are obtained using that 
$Q_\cE$ and $Q_\cM$ constitute a morphism of weak distributive laws. These
functors render commutative the outer square in the following diagram of
2-cells in $\SM$,
$$
\xymatrix{
f_q\cE\cM \ar[rr]^-{Q_\cE \cM}\ar[ddd]_-{f_q\varrho}\ar[rd]^-{f_q ii}&&
\cE'f_q\cM\ar[rr]^-{\cE' Q_\cM}\ar[d]^-{i'f_q i}&&
\cE'\cM'f_q\ar[ld]_-{i'i'f_q}\ar[ddd]^-{\varrho' f_q}\\
&f_q i(\cE)i(\cM)\ar[r]^-{{\widetilde Q}_\cE i(\cM)}
\ar[d]_-{f_q {\widetilde \varrho}} 
&i'(\cE')f_qi(\cM)\ar[r]^-{i'(\cE'){\widetilde Q}_\cM}
&i'(\cE')i'(\cM')f_q\ar[d]^-{{\widetilde\varrho}' f_q}&\\
&f_q i(\cM)i(\cE)\ar[r]^-{{\widetilde Q}_\cM i(\cE)}
&i'(\cM')f_qi(\cE)\ar[r]^-{i'(\cM'){\widetilde Q}_\cE}
&i'(\cM')i'(\cE')f_q&\\
f_q\cM\cE \ar[rr]^-{Q_\cM \cE}\ar[ru]^-{f_qii}&&
\cM'f_q\cE\ar[rr]^-{\cM' Q_\cE}\ar[u]_-{i'f_q i}&&
\cM'\cE'f_q\ar[lu]_-{i'i'f_q}
}
$$
where $f_q$ is the 1-cell corresponding to the function $q$, see
\ref{as:SetMat}. Since the 2-cell $f_q ii$ appearing in the upper left corner
is epi, this proves that also the inner square of the diagram commutes; that
is, $S(Q_\cE,Q_\cM):=({\widetilde Q}_\cE,{\widetilde Q}_\cM)$ is a morphism of
weak distributive laws. Moreover, there is a commutative square of morphisms
of weak distributive laws:
$$
\xymatrix{
(\varrho:\cE\cM\to \cM\cE)\ar[rr]^-{(i,i)}\ar[d]_-{(Q_\cE,Q_\cM)}&&
({\widetilde\varrho}:i(\cE)i(\cM)\to i(\cM)i(\cE))
\ar[d]^-{({\widetilde Q}_\cE,{\widetilde Q}_\cM)}\\
(\varrho':\cE'\cM'\to \cM'\cE')\ar[rr]^-{(i',i')}&&
({\widetilde\varrho}':i'(\cE')i'(\cM')\to i'(\cM')i'(\cE'))
}
$$
If $\varrho:\cE\cM\to \cM\cE$ is an injective weak distributive law in $\SM$,
then the functors $\cE \stackrel i \rightarrow \cM_\varrho\cE \stackrel i
\leftarrow \cM$ in Proposition \ref{prop:functor_i} factorize through
isomorphisms $i:\cE \to i(\cE)$  and $i:\cM \to i(\cM)$. By the above
considerations, these amount to an isomorphism $(\varrho:\cE\cM\to \cM\cE)\to
SJ(\varrho:\cE\cM\to \cM\cE)$ of weak distributive laws in $\SM$, natural with
respect to morphisms of weak distributive laws. The counit $\varepsilon:SJ\to
1$ of the adjunction is its inverse.
The unit $\eta: 1\to JS$ is given by the morphism $(i:\cE\to i(\cE),i:\cM\to
i(\cM))$ of weak distributive laws, natural in the arbitrary weak distributive
law $\varrho: \cE\cM\to \cM\cE$. This morphism is taken by $S$ to the identity
morphism $(i(\cE) \to i(\cE), i(\cM) \to i(\cM))$ which is equal to the value
of $\varepsilon S$ at the same object. Finally, at an injective weak
distributive law $\varrho:\cE\cM\to \cM\cE$, $\eta J$ is the isomorphism $(\cE
\to i(\cE), \cM \to i(\cM))$, with the inverse $J\varepsilon$.  
\end{proof}

Let us stress that, for an arbitrary weak distributive law $\varrho:\cE\cM\to
\cM\cE$, the injective weak distributive law $S(\varrho:\cE\cM\to \cM\cE)$
constructed in the proof of Proposition \ref{prop:retract} is still {\em
weak}. This can be seen by noting that the associated idempotent 2-cell
$i(\cM)i(\cE)\to i(\cM)i(\cE)$ in $\SM$ (cf. \ref{as:wdl}) is given by
functions   
$$
(i(m),i(e))\mapsto (i(1_c \rt m),i(e\lt 1_a)),
$$
for any morphisms $e:a \emor b$ in $\cE$ and $m:b \mmor c$ in $\cM$ -- which
are not identity functions unless $b=1_b.1_b$.
On the other hand, by the first identity in \eqref{eq:rt_1}, by the
second identity in \eqref{eq:lt_1} and by Lemma \ref{lem:rho_commutes},
$$
i(1_c \rt m)i(e\lt 1_a)=
i(m)i(1_b\rt 1_b)i(1_b\lt 1_b)i(e)=
i(m)i(1_b)i(1_b)i(e)=
i(m)i(e).
$$ 
Thus the product categories $i(\cM)_{S(\varrho)}i(\cE)$ and $\cM_\varrho \cE$
coincide. 

%%%%%%%%%%%%%%%%%%%%%%%%%%%%%%%%  SEC 2 %%%%%%%%%%%%%%%%%%%%%%%%%%%%%%%%%%%

\section{Bilinear factorization systems}
\label{sec:demifact}
 
\begin{aussage} {\bf Bilinear factorization system.} 
\label{as:ds_fact}
By a bilinear
  factorization system on a small category $\cC$ we mean subcategories
\footnote{Throughout, we use the term {\em subcategory} $\cE\hookrightarrow
  \cC$ `up-to isomorphism' -- i.e. in the loose sense that there is an
  identity-on-objects functor $i:\cE\to \cM$ which acts {\em injectively} on
  the morphisms -- we do not require $i$ to be an identity map on the
  morphisms.} 
$\cE\stackrel{i}{\hookrightarrow}
\cC \stackrel{i}{\hookleftarrow} 
\cM$,
  both containing all identity morphisms, together with a 2-cell $\delta:
  \cC\to \cM\cE$ in $\SM$,
$$
\cC(a,b)\to \sum_x \cM(x,b) \x \cE(a,x),\qquad 
g\mapsto \big( \mu(g):F(g) \mmor b, \epsilon(g):a\emor F(g)\big),
$$
providing a section for $\cM\cE\stackrel {ii} \hookrightarrow
\cC\cC\stackrel{\circ}{\to}\cC$ and rendering commutative the following
diagrams.  
$$
\xymatrix@R=10pt{
\cM\cC\ar[d]^-{i\cC} \ar[r]^-{\cM\delta}&
\cM\cM\cE \ar[dd]_-{\circ \cE}\\
\cC\cC\ar[d]^-{\circ}&\\
\cC\ar[r]^-{\delta}&
\cM\cE
}\qquad 
\xymatrix@R=10pt{
\cC\cE\ar[d]^-{\cC i}\ar[r]^-{\delta\cE} &\cM\cE\cE\ar[dd]_-{\cM\circ}\\
\cC\cC\ar[d]^-{\circ}&\\
\cC\ar[r]^-{\delta}&
\cM\cE
}
$$
In other words, for all morphisms $e:a\emor b$ in $\cE$, $g:b\to c$ in $\cC$
and $m:c\mmor d$ in $\cM$, there is a factorization $g=i\mu(g)i\epsilon(g)$
and the `$\cE$-, and $\cM$-linearity conditions'
\begin{eqnarray}
&\epsilon(i(m)g)=\epsilon(g)\quad\ \qquad
&\epsilon(gi(e))=\epsilon(g)e \label{eq:e_bilin}\\
&\mu(i(m)g)=m\mu(g)\qquad
&\mu(gi(e))=\mu(g) \label{eq:m_bilin}
\end{eqnarray}
hold. 
\end{aussage}

Clearly, any strict factorization system is bilinear. We could not find in the
literature, however, any non-strict factorization system which is bilinear. In
order to show that such factorization systems do exist, let stand here some
simple examples. 

\begin{example}\label{ex:order}
Let ${\mathcal P}$ be a {\em proset bounded from below}; that
is, a set equipped with a reflexive and transitive relation $\leq$ and an
element $\perp\in {\mathcal P}$ such that $\perp \leq a$ for all $a\in
{\mathcal P}$. For example, ${\mathcal P}$ can be the set of small sets with
$\leq$ denoting subsets and $\perp$ being the empty set. As another example,
${\mathcal P}$ can be chosen to be the set of positive integers with $p\leq q$
whenever $p$ divides $q$, in which case $\perp$ stands for the positive
integer $1$. 

Consider an associated category (in fact a groupoid) whose objects are the
elements of ${\mathcal P}$; and in which there is precisely one morphism
between any pair of objects. Let $\cE$ (resp. $\cM$) be the subcategory that
contains a morphism $p\to q$ if and only if $q\leq p$ (resp. $p\leq
q$). Clearly, both $\cE$ and $\cM$ contain all identity morphisms. But this is
not a strict factorization system as there is a morphism $p\emor l \mmor q$
whenever $l\leq p$ and $l\leq q$. This can be made a bilinear factorization
system, however, via $\epsilon(p\to q)=p \emor \perp$ and $\mu(p\to q)=\perp
\mmor q$. 
\end{example}

\begin{example}\label{ex:groupoid}
Let $\cC$ be a category with two objects $1$ and $2$, such that there is an
isomorphism $f:1\to 2$. Let $\cE$ be the subcategory containing all morphisms
in $\cC(1,1)\cup \cC(1,2)\cup \cC(2,2)$ and let $\cM$ be the subcategory
containing the identity morphisms $1_1$ and $1_2$ and the morphism
$f^{-1}:2\to 1$. Both $\cE$ and $\cM$ contain all identity morphisms. But this
is not a strict factorization system as $f^{-1}fg=g=1_1g$ for all $g\in
\cC(1,1)$. However, it is a bilinear factorization system via  
$$
\begin{array}{ll}
\epsilon(g)=\epsilon(fg)=fg\qquad \qquad 
&\epsilon(fgf^{-1})=\epsilon(gf^{-1})=fgf^{-1}\\
\mu(g)=\mu(gf^{-1})=f^{-1}\qquad \qquad 
&\mu(fg)=\mu(fgf^{-1})=1_2,
\end{array}
$$
for all $g\in \cC(1,1)$.
\end{example}

\begin{example}\label{ex:non_triv}
Consider a category $\cC$ with three objects $1$, $2$ and $3$ and morphism set
generated by three morphisms $f:1\to 2$, $p:2\to 3$ and $q:3\to 1$, modulo the
relations $qpf=1_1$ and $fqp=1_2$. (That is, the non-identity morphisms in
$\cC$ are $f:1\to 2$, $p:2\to 3$, $q:3\to 1$, $qp=f^{-1}:2\to 1$, $fq:3\to 2$,
$pf:1\to 3$ and $pfq:3\to 3$.) 
Let $\cE$ and $\cM$ be the subcategories with the same objects objects $1$,
$2$ and $3$; whose non-identity morphisms are $\{f,fq\}$ and
$\{p,qp=f^{-1}\}$, respectively. This is not a strict factorization system
since $f^{-1}f=1_1=1_11_1$. It is, however, a bilinear factorization system
via
$$
\begin{array}{ll}
\epsilon(1_1)=\epsilon(f)=\epsilon(pf)=f\quad
&\epsilon(q)=\epsilon(fq)=\epsilon(pfq)=fq\\
\epsilon(1_2)=\epsilon(p)=\epsilon(qp=f^{-1})=1_2\quad
&\epsilon(1_3)=1_3\\
\mu(1_1)=\mu(q)=\mu(qp=f^{-1})=f^{-1}\quad 
&\mu(p)=\mu(pf)=\mu(pfq)=p\\
\mu(1_2)=\mu(f)=\mu(fq)=1_2\quad 
&\mu(1_3)=1_3.
\end{array}
$$
\end{example}

\begin{aussage} {\bf Any weak distributive law in $\SM$ determines a bilinear
    factorization system.} \label{as:dl>fac}
Without loss of generality, we may take an injective weak distributive law
$\varrho:\cE\cM \to \cM\cE$. As if for it is not injective, we can replace it
with its image under the functor $S$ in Proposition \ref{prop:retract}. Thus we
have subcategories $\cE \stackrel{i}{\hookrightarrow} \cM_\varrho \cE
\stackrel{i}{\hookleftarrow} \cM$ which contain all identity morphisms. We put  
\begin{eqnarray*}
&\epsilon\big(
\xymatrix{
a\ar@{->>}[r]^-{e\lt 1_a}&
1_b.1_b \ \ar@{>->}[r]^-{1_c \rt m}&
c}\big):=&\xymatrix{
\ \ a\ \ \ \ar@{->>}[r]^-{e\lt 1_a}&
1_b.1_b}\qquad \textrm{and}\\
&\mu\big(
\xymatrix{
a\ar@{->>}[r]^-{e\lt 1_a}&
1_b.1_b \ \ar@{>->}[r]^-{1_c \rt m}&
c}\big):=&\xymatrix{
1_b.1_b \ \ar@{>->}[r]^-{1_c \rt m}&
c.}
\end{eqnarray*}
Use \eqref{eq:inif} to see that, for any morphism 
$g=\big(
\xymatrix{
a\ar@{->>}[r]^-{e\lt 1_a}&
1_b.1_b \ \ar@{>->}[r]^-{1_c \rt m}&
c}\big)$ in $\cM_\varrho\cE$,
$$
i\mu(g)i\epsilon(g)=
i(1_c \rt m)i(e\lt 1_a)=
\big(
\xymatrix{
a\ar@{->>}[r]^-{e\lt 1_a}&
1_b.1_b \ \ar@{>->}[r]^-{1_c \rt m}&
c}\big)=g.
$$
Since any morphism in $\cM_\varrho \cE$ can be written in the form $i(m)i(f)$
in terms of some (non-uniquely chosen) morphisms $f:b\emor c$ in $\cE$ and
$m:c\mmor d$ in $\cM$, the `$\cE$-linearity conditions' follow for any
morphism $e:a\emor b$ in $\cE$ by  
\begin{eqnarray*}
&&\epsilon\big(i(m)i(f)i(e)\big)=
\epsilon\big(i(m)i(fe)\big)=
fe \lt 1_a =
(f\lt 1_b)e=
\epsilon\big(i(m)i(f)\big)e; \\
&&\mu\big(i(m)i(f)i(e)\big)=
\mu\big(i(m)i(fe)\big)=
1_d \rt m =
\mu\big(i(m)i(f)\big).
\end{eqnarray*}
The third equality in the first line follows by the second condition in
\eqref{eq:lt_1}. The `$\cM$-linearity conditions' are verified symmetrically.
\end{aussage}

\begin{aussage} {\bf Any bilinear factorization system determines a weak
  distributive law in $\SM$.} \label{as:fac>dl}
For a bilinear factorization system $\cE\stackrel i \hookrightarrow
\cC\stackrel i \hookleftarrow \cM$ with $\delta : \cC\to \cM\cE$, we put 
$$
\varrho:=\big(
\xymatrix{
\cE\cM\ \ar@{^{(}->}[r]^{ii}&
\cC\cC\ar[r]^-\circ&
\cC\ar[r]^-\delta&
\cM\cE
}
\big).
$$
It renders commutative the four diagrams in \ref{as:wdl}:
$$
\xymatrix@R=16pt @C=16pt{
\cE\cE\cM\ar[rrr]^-{\circ \cM}\ar[d]_-{\cE ii}&&&
\cE\cM\ar[d]^-{ii}\\
\cE\cC\cC\ar[rr]^-{i\cC\cC}\ar[d]_-{\cE\circ}&&
\cC\cC\cC \ar[r]^-{\circ \cC}\ar@/^1.5pc/[ddd]^-{\cC\circ}&
\cC\cC\ar[ddd]^-\circ\\
\cE\cC\ar[d]_-{\cE \delta}\ar@{=}[rrd]\ar@{}[rd]|(.7){(*)}&&&\\
\cE\cM\cE\ar[r]_-{\cE ii}\ar[d]_-{ii\cE}&
\cE\cC\cC\ar[r]_-{\cE \circ}\ar[d]^-{i\cC\cC}&
\cE\cC\ar[d]_-{i\cC}&\\
\cC\cC\cE\ar[r]^-{\cC\cC i}\ar[d]_-{\circ \cE}&
\cC\cC\cC\ar[r]^-{\cC \circ}\ar[d]^-{\circ \cC}&
\cC\cC\ar[r]^-{\circ}&
\cC\ar[dd]^-{\delta}\\
\cC\cE\ar[r]^-{\cC i}\ar[d]_-{\delta \cE}&
\cC\cC\ar[urr]_-{\circ}\ar@{}[rrd]|-{(**)}&&\\
\cM\cE\cE\ar[rrr]_-{\cM \circ}&&&
\cM\cE}
\quad\qquad
\xymatrix@R=16pt @C=12.5pt{
\cE\cM\cM\ar[rrr]^-{\cE\circ}\ar[d]_-{ii\cM}&&&
\cE\cM\ar[d]^-{ii}\\
\cC\cC\cM\ar[rr]^-{\cC\cC i}\ar[d]_-{\circ \cM}&&
\cC\cC\cC \ar[r]^-{\cC\circ}\ar@/^1.5pc/[ddd]^-{\circ \cC}&
\cC\cC\ar[ddd]^-\circ\\
\cC\cM\ar[d]_-{\delta \cM}\ar@{=}[rrd]\ar@{}[rd]|(.7){(*)}&&&\\
\cM\cE\cM\ar[r]_-{ii\cM}\ar[d]_-{\cM ii}&
\cC\cC\cM\ar[r]_-{\circ \cM}\ar[d]^-{\cC\cC i}&
\cC\cM\ar[d]_-{\cC i}&\\
\cM\cC\cC\ar[r]^-{i\cC\cC}\ar[d]_-{\cM\circ}&
\cC\cC\cC\ar[r]^-{\circ \cC}\ar[d]^-{\cC\circ}&
\cC\cC\ar[r]^-{\circ}&
\cC\ar[dd]^-{\delta}\\
\cM\cC\ar[r]^-{i\cC}\ar[d]_-{\cM \delta}&
\cC\cC\ar[urr]_-{\circ}\ar@{}[rrd]|-{(**)}&&\\
\cM\cM\cE\ar[rrr]_-{\circ \cE}&&&
\cM\cE}
$$
The triangles labelled by $(*)$ commute since $\delta$ is a section of
$\circ.ii$ and the polygons $(**)$ commute by the $\cE$-linearity and by the
$\cM$-linearity of $\delta$, respectively. The other regions commute since the
$i$s are functors, since the composition in $\cC$ is associative and by the
middle four interchange law in $\SM$. Moreover, denoting by ${\bf 1}$ the unit
of a monad in $\SM$ (describing the identity morphisms in the corresponding
category), also the following diagrams commute. 
$$
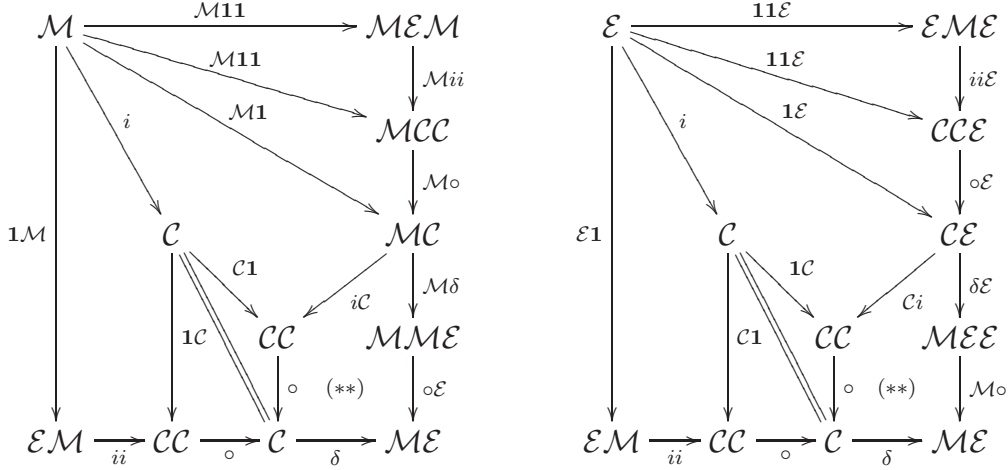
\xymatrix @C=19pt{
\cM\ar[rrr]^-{\cM {\bf 1}{\bf 1}}\ar[dddd]_-{{\bf 1} \cM}
\ar[drrr]^-{\cM{\bf 1}{\bf 1}}\ar[rrrdd]^-{\cM {\bf 1}}\ar[rdd]^-i&&&
\cM\cE\cM\ar[d]^-{\cM ii}\\
&&&\cM\cC\cC\ar[d]^-{\cM\circ}\\
&\cC\ar[rd]^-{\cC{\bf 1}}\ar[dd]^{{\bf 1}\cC}\ar@{=}[rdd]&&
\cM\cC\ar[d]^-{\cM\delta}\ar[ld]^-{i\cC}\\
&&\cC\cC\ar[d]^-{\circ}\ar@{}[rd]|-{(**)}&
\cM\cM\cE\ar[d]^-{\circ \cE}\\
\cE\cM\ar[r]_-{ii}&
\cC\cC\ar[r]_-\circ &
\cC\ar[r]_-\delta&
\cM\cE
}\quad \qquad
\xymatrix @C=19pt{
\cE\ar[rrr]^-{{\bf 1}{\bf 1}\cE}\ar[dddd]_-{\cE{\bf 1}}
\ar[drrr]^-{{\bf 1}{\bf 1}\cE}\ar[rrrdd]^-{{\bf 1}\cE}\ar[rdd]^-i&&&
\cE\cM\cE\ar[d]^-{ii\cE}\\
&&&\cC\cC\cE\ar[d]^-{\circ \cE}\\
&\cC\ar[rd]^-{{\bf 1}\cC}\ar[dd]^{\cC{\bf 1}}\ar@{=}[rdd]&&
\cC\cE\ar[d]^-{\delta \cE}\ar[ld]^-{\cC i}\\
&&\cC\cC\ar[d]^-{\circ}\ar@{}[rd]|-{(**)}&
\cM\cE\cE\ar[d]^-{\cM\circ}\\
\cE\cM\ar[r]_-{ii}&
\cC\cC\ar[r]_-\circ &
\cC\ar[r]_-\delta&
\cM\cE
}
$$
Here again, the same commutative regions $(**)$ appear. The other regions
commute since the $i$s are functors, by triviality of compositions with
identity morphisms and by the middle four interchange law in $\SM$.
\end{aussage} 

\begin{example}
The bilinear factorization system in Example \ref{ex:order} determines a weak
distributive law $\varrho:\cE\cM\to\cM\cE$ in $\SM$ as follows. For $a\leq b$
and $c\leq b$ in ${\mathcal P}$, consider $m:a\mmor b$ and $e:b\emor c$. Then
$e.m=\perp$; $e\lt m$ is {\em the} morphism $a\emor \perp$ and $e\rt m$ is {\em
  the} morphism $\perp\mmor c$.
\end{example}

\begin{example}
The bilinear factorization system in Example \ref{ex:groupoid} determines a
weak distributive law $\varrho:\cE\cM\to \cM\cE$ in $\SM$ as follows. For all
$g\in \cC(1,1)$, 
$$
\begin{array}{ll}
\varrho(g,1_1)=\big(\xymatrix{1\ar@{->>}[r]^-{fg}&2 \ 
  \ar@{>->}[r]^-{f^{-1}}&1}\big);\qquad
&\varrho(g,f^{-1})=\big(\xymatrix{2\ar@{->>}[r]^-{fgf^{-1}}&2 \ 
  \ar@{>->}[r]^-{f^{-1}}&1}\big);\\ 
\varrho(fgf^{-1},1_2)=\big(\xymatrix{2\ar@{->>}[r]^-{fgf^{-1}}&2 \ 
  \ar@{>->}[r]^-{1_2}&2}\big);\qquad
&\varrho(fg,1_1)=\big(\xymatrix{1\ar@{->>}[r]^-{fg}&2 \ 
  \ar@{>->}[r]^-{1_2}&2}\big);\\
\varrho(fg,f^{-1})=\big(\xymatrix{2\ar@{->>}[r]^-{fgf^{-1}}&2 \ 
  \ar@{>->}[r]^-{1_2}&2}\big).\qquad &
\end{array}
$$
\end{example}

\begin{example}
The bilinear factorization system in Example \ref{ex:non_triv} determines a
weak distributive law $\varrho:\cE\cM\to \cM\cE$ in $\SM$ as follows. 
$$
\begin{array}{ll}
\varrho(1_1,1_1)=\big(\xymatrix{1\ar@{->>}[r]^-{f}&2 \ 
  \ar@{>->}[r]^-{f^{-1}}&1}\big);\qquad
&\varrho(f,1_1)=\big(\xymatrix{1\ar@{->>}[r]^-{f}&2 \ 
  \ar@{>->}[r]^-{1_2}&2}\big);\\
\varrho(1_2,1_2)=\big(\xymatrix{2\ar@{->>}[r]^-{1_2}&2 \ 
  \ar@{>->}[r]^-{1_2}&2}\big);\qquad
&\varrho(p,1_2)=\big(\xymatrix{2\ar@{->>}[r]^-{1_2}&2 \ 
  \ar@{>->}[r]^-{p}&3}\big);\\
\varrho(f^{-1},1_2)=\big(\xymatrix{2\ar@{->>}[r]^-{1_2}&2 \ 
  \ar@{>->}[r]^-{f^{-1}}&1}\big);\qquad
&\varrho(1_3,1_3)=\big(\xymatrix{3\ar@{->>}[r]^-{1_3}&3 \ 
  \ar@{>->}[r]^-{1_3}&3}\big);\\
\varrho(q,1_3)=\big(\xymatrix{3\ar@{->>}[r]^-{fq}&2 \ 
  \ar@{>->}[r]^-{f^{-1}}&1}\big);\qquad
&\varrho(fq,1_3)=\big(\xymatrix{3\ar@{->>}[r]^-{fq}&2 \ 
  \ar@{>->}[r]^-{1_2}&2}\big);\\
\varrho(1_3,p)=\big(\xymatrix{2\ar@{->>}[r]^-{1_2}&2 \ 
  \ar@{>->}[r]^-{p}&3}\big);\qquad
&\varrho(q,p)=\big(\xymatrix{2\ar@{->>}[r]^-{1_2}&2 \ 
  \ar@{>->}[r]^-{f^{-1}}&1}\big);\\
\varrho(fq,p)=\big(\xymatrix{2\ar@{->>}[r]^-{1_2}&2 \ 
  \ar@{>->}[r]^-{1_2}&2}\big);\qquad
&\varrho(1_1,f^{-1})=\big(\xymatrix{2\ar@{->>}[r]^-{1_2}&2 \ 
  \ar@{>->}[r]^-{f^{-1}}&1}\big);\\
\varrho(f,f^{-1})=\big(\xymatrix{2\ar@{->>}[r]^-{1_2}&2 \ 
  \ar@{>->}[r]^-{1_2}&2}\big).
\end{array}
$$
\end{example}

Our next aim is to prove that the constructions in \ref{as:dl>fac} and
\ref{as:fac>dl} are mutual inverses in an appropriate sense.

\begin{theorem} \label{thm:iso_T}
For any bilinear factorization system $\cE \stackrel i \hookrightarrow
  \cC \stackrel i \hookleftarrow \cM$, $\cC$ is isomorphic to the product
  category $\cM_\varrho \cE$ corresponding to the weak distributive law
  $\varrho$ in \ref{as:fac>dl}.
\end{theorem}

\begin{proof}
The objects of both categories $\cC$ and $\cM_\varrho \cE$ coincide by
construction. As explained in \ref{as:prod_cat}, the morphisms in 
$\cM_\varrho \cE$ are of the form
$$
\xymatrix{
a \ar@{->>}[rr]^-{\epsilon i(e)}&& F(1_b)\  \ar@{>->}[rr]^-{\mu i(m)}&&c,
}
$$
for morphisms $e:a\emor b$ in $\cE$ and $m:b\mmor c$ in $\cM$, where the same
notations are used as in \ref{as:ds_fact}. The stated isomorphism is given by 
$$
T:\cM_\varrho \cE \to \cC,\quad 
\big(
\xymatrix{
a \ar@{->>}[r]^-{\epsilon i(e)}& F(1_b)\  \ar@{>->}[r]^-{\mu i(m)}&c
}
\big)\ \mapsto \ 
\big(
\xymatrix{
a \ar[r]^-{i\epsilon i(e)}& F(1_b) \ar[r]^-{i\mu i(m)}&c
}
\big)
$$
and its inverse
$$
T^{-1}: \cC\to \cM_\varrho \cE,\qquad \qquad \qquad
\big(
\xymatrix{
a\ar[r]^-{g}&c
}
\big)\mapsto
\big(
\xymatrix{
a \ar@{->>}[r]^-{\epsilon(g)}& F(g)\  \ar@{>->}[r]^-{\mu(g)}&c
}
\big).
$$
$T^{-1}(g)$ is a morphism in $\cM_\varrho \cE$ as needed, since by the
bilinearity properties of $\epsilon$ and $\mu$, 
\begin{equation}\label{eq:eie}
\epsilon i\epsilon(g)=
\epsilon\big(i\mu(g)\,  i\epsilon(g)\big) =
\epsilon(g); \quad
\mu i\mu(g) =
\mu\big(i\mu(g) \, i\epsilon(g)\big) =
\mu(g).
\end{equation}
Evidently, $T^{-1}$ preserves identity morphisms. It preserves compositions as
well, since for any morphisms $g:a\to c$ and $h:c\to d$ in $\cC$, 
\begin{eqnarray*}
T^{-1}(h)T^{-1}(g)&=&
\big(\mu(h),\epsilon(h)\big) \big(\mu(g),\epsilon(g)\big)=
\big(\mu(h)\mu(i\epsilon(h)\, i\mu(g)),\epsilon(i\epsilon(h)\,
i\mu(g))\epsilon(g) \big)\\
&=&\big(\mu(i\mu(h)\, i\epsilon(h)\, i\mu(g)\, i\epsilon(g)), 
\epsilon(i\mu(h)\, i\epsilon(h)\, i\mu(g)\, i\epsilon(g))\big)\\
&=&\big(\mu(hg),\epsilon(hg)\big)=
T^{-1}(hg).
\end{eqnarray*}
The third equality follows by the bilinearity of $\mu$ and $\epsilon$, see
\eqref{eq:e_bilin} and \eqref{eq:m_bilin}, and the fourth equality follows by
the factorization property. 

Composites of $T$ and $T^{-1}$ in both orders yield identity functors by
\eqref{eq:eie} and by the factorization property, respectively:
\begin{eqnarray*}
&&T^{-1}T\big(
\xymatrix @C=15pt{
a \ar@{->>}[r]^-{\epsilon i(e)}& F(1_b)\  \ar@{>->}[r]^-{\mu i(m)}&c
}
\big)=
T^{-1}\big(
\xymatrix{
a \ar[r]^-{i\epsilon i(e)}& F(1_b) \ar[r]^-{i\mu i(m)}&c
}
\big)=
\big(
\xymatrix @C=15pt{
a \ar@{->>}[r]^-{\epsilon i(e)}& F(1_b)\  \ar@{>->}[r]^-{\mu i(m)}&c
}
\big)\\
&&TT^{-1}\big(
\xymatrix{
a\ar[r]^-{g}&c
}
\big)=
T\big(
\xymatrix{
a \ar@{->>}[r]^-{\epsilon(g)}& F(g)\  \ar@{>->}[r]^-{\mu(g)}&c
}
\big)=
\big(
\xymatrix{
a\ar[r]^-{g}&c
}
\big).
\end{eqnarray*}
\end{proof}

\begin{aussage}{\bf Morphisms of bilinear factorization systems.}
A morphism $(\cE \stackrel i \hookrightarrow \cC \stackrel i \hookleftarrow
\cM)\to (\cE' \stackrel {i'}\hookrightarrow \cC'  \stackrel {i'}\hookleftarrow
\cM')$ of bilinear factorization systems is by definition a functor
$Q:\cC\to \cC'$ which restricts (along the inclusions $i$) to functors
$Q_\cE:\cE\to \cE'$ and $Q_\cM:\cM\to \cM'$ such that  
$$
Q_\cE(\epsilon(g))=\epsilon'(Q(g))\qquad 
Q_\cM(\mu(g))=\mu'(Q(g)),
$$
for any morphism $g$ in $\cC$ and $\mu$ and $\epsilon$ as in \ref{as:ds_fact}.

Bilinear factorization systems and their morphisms constitute a category. 
\end{aussage}

\begin{theorem}\label{thm:fac_eq_wdl}
The category of injective weak distributive laws in $\SM$ and the category of
bilinear factorization systems are equivalent.
\end{theorem}

\begin{proof}
We construct mutually inverse equivalence functors with respective object maps
in \ref{as:dl>fac} and \ref{as:fac>dl}.

A functor from the category of bilinear factorization systems to the
category of (injective) weak distributive laws is $\SM$ is given by 
$$
\big(
\xymatrix @=12pt{
(\cE \stackrel i \hookrightarrow \cC\stackrel i \hookleftarrow
  \cM)\ar[r]^-{Q}& 
(\cE' \stackrel {i'} \hookrightarrow \cC'\stackrel {i'} \hookleftarrow
  \cM')}\big) \mapsto 
\big(
\xymatrix @C=10pt{
(\cE\cM \ar[r]^-\varrho&\cM\cE) \ar[rr]^-{(Q_\cE,Q_\cM)}&&
(\cE'\cM' \ar[r]^-{\varrho'}&\cM'\cE')}
\big),
$$
where the weak distributive laws $\varrho$ and $\varrho'$ are constructed as
in \ref{as:fac>dl} and the functors $Q_\cE:\cE\to \cE'$ and $Q_\cM:\cM\to
\cM'$ are restrictions of $Q$. For the functor $T^{-1}$ in Theorem
\ref{thm:iso_T}, $T^{-1}i$ is equal to the functor associated in Proposition
\ref{prop:functor_i} to $\varrho$. Hence it follows by Theorem \ref{thm:iso_T}
that $\varrho$ is an injective weak distributive law. Since $Q$ is a morphism
of bilinear factorization systems, using the notations in \ref{as:wdl_SM} and
\ref{as:ds_fact} we obtain, for any compatible morphisms $e$ in $\cE$ and $m$
in $\cM$, 
\begin{eqnarray*}
Q_\cE(e\lt m)
&=& Q_\cE(\epsilon(i(e)\, i(m)))
=\epsilon'(Q(i(e)\, i(m)))
=\epsilon'(Qi(e)\, Qi(m)) \\
&=&\epsilon'(i'Q_\cE(e)\, i'Q_\cM(m))
=Q_\cE(e)\lt Q_\cM(m)
\end{eqnarray*}
and symmetrically, $Q_\cM(e\rt m)=Q_\cE(e)\rt Q_\cM(m)$. This proves that
$(Q_\cE,Q_\cM)$ is a morphism of weak distributive laws.

A functor from the category of injective weak distributive laws in $\SM$ to
the category of bilinear factorization systems
$$
\big(\!\!
\xymatrix @C=10pt{
(\cE\cM \ar[r]^-\varrho&\cM\cE) \ar[rr]^-{(Q_\cE,Q_\cM)}&&
(\cE'\cM' \ar[r]^-{\varrho'}&\cM'\cE')}
\!\! \big) \!\!\mapsto \!\!
\big(\!\!
\xymatrix @=12pt{
(\cE \hookrightarrow \cM_\varrho \cE \hookleftarrow \cM)\ar[r]^-{Q}&
(\cE' \hookrightarrow \cM'_{\varrho'} \cE'\hookleftarrow \cM')}\!\!\big)
$$
is defined by the functor $Q:\cM_\varrho\cE\to \cM'_{\varrho'}\cE'$,
\begin{eqnarray} \label{eq:Q}
\big(
\xymatrix{
a\ar@{->>}[r]^-{e\lt 1_a}&1_b.1_b\ \ar@{>->}[r]^-{1_c\rt m}&c}\big) &\mapsto&
\big(
\xymatrix @C=40pt{
q(a)\ar@{->>}[r]^-{Q_\cE(e\lt 1_a)}&q(1_b.1_b)\ 
\ar@{>->}[r]^-{Q_\cM(1_c\rt m)}&q(c)
}\big)=\\
&&\big(
\xymatrix @C=40pt{
q(a)\ar@{->>}[r]^-{Q_\cE(e)\lt 1_{q(a)}}&1_{q(b)}.1_{q(b)}
\ \ar@{>->}[r]^-{1_{q(c)}\rt Q_\cM(m)}&q(c)}
\big),\nonumber
\end{eqnarray}
where $q$ is the common object map of the functors $Q_\cE$ and $Q_\cM$ and the
notations in \ref{as:wdl_SM} are used. In order to see that $Q$ is a 
morphism of bilinear factorization systems, note that for the functors $i$
and $i'$ as in Proposition \ref{prop:functor_i} and for any morphism $e:a\emor
b$ in $\cE$, 
\begin{eqnarray*}
i'Q_\cE(e)
&=& \big(\xymatrix @C=40pt {
q(a)\ar@{->>}[r]^-{Q_\cE(e)\lt 1_{q(a)}}&1_{q(b)}.1_{q(b)} \ 
\ar@{>->}[r]^-{1_{q(b)}\rt 1_{q(b)}}&q(b)}\big)\\
&=&Q\big(\xymatrix{
a\ar@{->>}[r]^-{e\lt 1_a}&1_b.1_b\ \ar@{>->}[r]^-{1_b\rt 1_b}&b}\big)
=Qi(e)
\end{eqnarray*}
and symmetrically, $i'Q_\cM(m)=Qi(m)$ for any morphism $m$ in $\cM$. Thus
$Q_\cE$ and $Q_\cM$ are restrictions of $Q$. Furthermore, for any morphism
$g=\big(\xymatrix{a\ar@{->>}[r]^-{e\lt 1_a}&1_b.1_b\ \ar@{>->}[r]^-{1_c \rt
m}&c}\big)$ in $\cM_\varrho \cE$, 
\begin{eqnarray*}
Q_\cE(\epsilon(g))
&=& Q_\cE(\xymatrix{a\ar@{->>}[r]^-{e\lt 1_a}& 1_b.1_b})
=\big(\xymatrix @C=40pt{q(a)\ar@{->>}[r]^-{Q_\cE(e)\lt 1_{q(a)}}&
1_{q(b)}.1_{q(b)}}\big)\\
&=&\epsilon'\big(\xymatrix @C=40pt{q(a)\ar@{->>}[r]^-{Q_\cE(e)\lt 1_{q(a)}}&
1_{q(b)}.1_{q(b)}\ \ar@{>->}[r]^-{1_{q(c)}\rt Q_\cM(m)}& q(c)}\big)
=\epsilon'(Q(g))
\end{eqnarray*}
and symmetrically, $Q_\cM(\mu(g))=\mu'(Q(g))$.

Iterating both functors above on the category of injective weak distributive
laws, we clearly obtain the identity functor. In the opposite order, a
morphism $(\cE \stackrel j \hookrightarrow \cC\stackrel j\hookleftarrow
\cM)\stackrel S\to  (\cE'\stackrel{j'}\hookrightarrow
\cC'\stackrel{j'}\hookleftarrow \cM')$ of bilinear factorization systems is
taken to $(\cE\stackrel i\hookrightarrow \cM_\varrho \cE\stackrel i
\hookleftarrow \cM)\stackrel Q\to (\cE'\stackrel{i'}\hookrightarrow 
\cM'_{\varrho'}\cE'\stackrel{i'} \hookleftarrow \cM')$, where the weak
distributive laws $\varrho$ and $\varrho'$ are of the form in \ref{as:fac>dl},
$\cM_\varrho\cE$ and $\cM'_{\varrho'}\cE'$ are the induced product categories
and the functors $i$ and $i'$ are as in Proposition
\ref{prop:functor_i}. The functor $Q$ is constructed from the restrictions
$S_\cE:\cE\to \cE'$ and $S_\cM:\cM\to \cM'$ of $S$ via \eqref{eq:Q}. The proof
of the theorem is completed by showing that the isomorphism $T$ in Theorem
\ref{thm:iso_T} induces a natural isomorphism between this composite functor
and the identity functor on the category of bilinear factorization
systems. The functor $T:\cM_\varrho\cE\to \cC$ (and hence also its inverse)
restricts to identity functors on $\cE$ and $\cM$. That is, for any morphism
$e:a\emor b$ in $\cE$,  
\begin{eqnarray*}
Ti(e)
&=& T\big(\xymatrix{a\ar@{->>}[r]^-{\epsilon j(e)}&F(1_b)\ 
\ar@{>->}[r]^-{\mu(1_b)} &b}\big)
=\big(\xymatrix{a\ar[r]^-{j\epsilon j(e)}&F(1_b)
\ar[r]^-{j\mu(1_b)} &b}\big)\\
&=& \big(\xymatrix{a \ar[r]^-{j(e)}&b\ar[r]^-{j\epsilon(1_b)}&F(1_b)
  \ar[r]^-{j\mu(1_b)}& b}\big)
=j(e),
\end{eqnarray*}
where the penultimate equality follows by the second identity in
\eqref{eq:e_bilin}. Symmetrically, $Ti(m)=j(m)$ for any morphism $m$ in
$\cM$. Moreover, for any morphism
$g=\big(\xymatrix{a\ar@{->>}[r]^-{\epsilon j(e)}&F(1_b)\ 
\ar@{>->}[r]^-{\mu j(m)}& c}\big)$ in $\cM_\varrho\cE$, 
\begin{eqnarray*}
\epsilon(T(g))
&=& \epsilon\big(\xymatrix @C=35pt{a\ar[r]^-{j\epsilon j(e)}&F(1_b)
\ar[r]^-{j\mu j(m)}& c}\big)
=\big(\xymatrix{a\ar@{->>}[r]^-{\epsilon j(e)}&F(1_b)}\big)
=\epsilon_{\cM_\varrho \cE}(g),
\end{eqnarray*}
where the penultimate equality follows by the bilinearity properties of
$\epsilon$; cf. \eqref{eq:eie}. Symmetrically, also
$\mu(T(g))=\mu_{\cM_\varrho \cE}(g)$. This proves that $T$ is an isomorphism
of bilinear factorization systems. 

Its naturality, i.e. the equality of functors $S T = T' Q$ is
immediate. 
\end{proof}

As a consequence of Proposition \ref{prop:retract} and Theorem
\ref{thm:fac_eq_wdl}, we obtain the following.

\begin{corollary} \label{cor:wdl_fac}
There is an adjunction whose counit is an isomorphism, between the category of
weak distributive laws in $\SM$ and the category of bilinear factorization
systems. 
\end{corollary} 

%%%%%%%%%%%%%%%%%%%%%%%%%%%%  SEC 3 %%%%%%%%%%%%%%%%%%%%%%%%%%%%%%%%%%%%%%%

\section{Strictly associative pseudoalgebras for the 2-monad $\ad2$ on $\Cat$}

\begin{aussage}{\bf The 2-monad $\ad2$ on $\Cat$.}
Recall (e.g. from Section 1 of \cite{KorTho}) that the 2-functor $\ad2$ on the
2-category $\Cat$ of categories; functors; and natural transformations, is
given as follows.  
For any category $\cC$, the objects of $\cC^2$ are morphisms in
$\cC$. Morphisms from $f:a\to x$ to $g:b\to y$ are pairs of morphisms $u:a\to
b$ and $v:x \to y$ such that the first diagram in
\begin{equation}\label{eq:K^1_mor}
\xymatrix{
a\ar[r]^-{u}\ar[d]_-{f}& b\ar[d]^-{g}\\
x\ar[r]_-{v}&y}
\qquad\qquad
\raise-.6cm\hbox{$F(u,v)=$}
\xymatrix{
F(a)\ar[r]^-{F(u)}\ar[d]_-{F(f)}&F(b)\ar[d]^-{F(g)}\\
F(x)\ar[r]_-{F(v)}&F(y)}
\qquad\qquad
\xymatrix{
F(a)\ar[r]^-{\omega_a}\ar[d]_-{F(f)}&G(a)\ar[d]^-{G(f)}\\
F(x)\ar[r]_-{\omega_x}&G(x)}
\end{equation}
commutes. For a functor $F:\cC \to \cD$, $F^2:\cC^2\to \cD^2$  takes the
morphism in the first figure to the morphism in $\cD^2$ in the second figure
of \eqref{eq:K^1_mor}.  
For a natural transformation $\omega:F \to G$ and an object $f:a\to x$ of
$\cC^2$ (i.e. a morphism $f$ in $\cC$), $\omega^2_f$ is the morphism in
$\cD^2$ depicted in the third figure of \eqref{eq:K^1_mor}. 

Multiplication of the 2-monad $\ad2$ is given at any category $\cC$ by the
functor $M_\cC:\cC^{2\x 2}\to \cC^2$ with the morphism map
\begin{equation}\label{eq:ad2_mp}
\xymatrix @R=2pt @C=20pt{
&&&a'\ar[rr]^-{u'}\ar[dd]_-{f'}&&b'\ar[dd]^-{g'}\\ 
&&&&&\\
&&&x'\ar[rr]_-{v'}&&y'\\
a\ar[rr]^-{u}\ar[dd]_-{f}\ar[rrruuu]^-{k}&&
b\ar[dd]^-{g}\ar[rrruuu]^-{l}&&&\\ 
&&&&&\\
x\ar[rr]_-{v} \ar[rrruuu]^(.7){m}|(.62){\color{white} III}&&
y\ar[rrruuu]_(.7){n}&&&}
\quad \raise-1cm\hbox{$\mapsto$} \quad
\xymatrix @=55pt{
a\ar[r]^-{k}\ar[d]_-{vf=gu}&
a'\ar[d]^-{v'f'=g'u'}\\
y\ar[r]_-{n}&y'}
\end{equation}
and the unit is given by the functor $I_\cC:\cC\to \cC^2$ with the morphism
map 
$$
 \raise-.6cm\hbox{$
\xymatrix{a\ar[r]^-f&x}$}
\quad \raise-.6cm\hbox{$\mapsto$} \quad
\xymatrix{
a\ar[r]^-{f}\ar[d]_-{1_a}&x\ar[d]^-{1_x}\\
a\ar[r]_-f&x\ .}
$$
\end{aussage}

\begin{aussage} {\bf (Strictly associative) pseudoalgebras of the 2-monad
    $\ad2$.}  \label{as:psalg}
Applying the definition of a pseudoalgebra (see e.g. \cite{Kelly_lax_alg}) to
the particular 2-monad $\ad2$, the following notion is obtained. It is given
by a category $\cC$, a functor $F:\cC^2\to \cC$ and natural isomorphisms
$\vartheta: FI_\cC\to 1_\cC$ and $\varphi:FM_\cC\to FF^2$ such that the
following diagrams (of natural transformations) commute.
$$
\xymatrix @=5pt{
F\ar@{=}[r]\ar@{=}[d]\ar@{=}[rrrrdddd]&
FM_{\cC} I_{\cC}^2  \ar[rrr]^-{\varphi I_{\cC}^2}&&&
FF^2I_{\cC}^2\ar@{=}[d]\\
FM_\cC I_{\cC^2}\ar[ddd]_-{\varphi I_{\cC^2}}&&&&
F(FI_{\cC})^2\ar[ddd]^-{F\vartheta^2}\\
&&&&\\
&&&&\\
FF^2I_{\cC^2}\ar@{=}[r]&FI_{\cC} F\ar[rrr]_-{\vartheta F}&&&F
}\qquad
\xymatrix @R=5.4pt @C=2pt{
FM_\cC M_\cC^2\ar[rr]^-{\varphi M_{\cC}^2}\ar@{=}[d]&&
FF^2 M_{\cC}^2  \ar@{=}[r]&
F(FM_{\cC})^2\ar[dd]^-{F\varphi^2}\\
FM_\cC M_{\cC^2}\ar[dd]_-{\varphi M_{\cC^2}}&&&\\
&&&F(FF^2)^2\ar@{=}[d]\\
FF^2M_{\cC^2}\ar@{=}[r]&FM_{\cC} F^{2\x 2}\ar[rr]_-{\varphi F^{2\x 2}}&&
FF^2F^{2\x 2}
}
$$
We say that a pseudoalgebra $(\cC,F,\vartheta,\varphi)$ is strictly
associative whenever $\varphi$ is the identity -- hence the natural
isomorphism $\vartheta:FI_\cC\to 1_\cC$ satisfies 
\begin{equation}\label{eq:associ_psalg}
F \vartheta^2=1_F =\vartheta F.
\end{equation}
Strict associativity of $F$ explicitly means the equality 
\begin{equation}\label{eq:F_associ}
\raise-.7cm\hbox{$F\left( \right .$}
\xymatrix{
F(f)\ar[r]^-{F(k,m)}\ar[d]_-{F(u,v)}&
F(f')\ar[d]^-{F(u',v')}\\
F(g)\ar[r]_-{F(l,n)}&
F(g')}
\raise-.7cm\hbox{$\left . \right)=F\left(\right .$}
\xymatrix{
a\ar[d]_-{gu}\ar[r]^-{k}&
a'\ar[d]^-{g'u'}\\
y\ar[r]_-{n}&
y'}
\raise-.7cm\hbox{$\left .\right)$}
\end{equation}
of morphisms in $\cC$, for any morphism in $\cC^{2\x 2}$ as on the left hand
side of \eqref{eq:ad2_mp}.

As explained in Section 2.2 of \cite{KorTho}, any pseudoalgebra
$(\cC,F,\vartheta,\varphi)$ of $\ad2$ can be modified such that it becomes
strictly unital, i.e. such that the natural isomorphism $\vartheta:FI_\cC\to
1_\cC$ becomes the identity. This modification certainly affects the
associativity natural isomorphism $\varphi:FM_\cC\to FF^2$ as well. 
Our functorial constructions later in this section, however, lead to strictly
associative but not necessarily strictly unital pseudoalgebras. We prefer not
to modify them to be strictly unital (spoiling perhaps their strict
associativity). As we will see, in our approach the unitality natural
isomorphism $\vartheta$ arises in a natural way.    
\end{aussage}

\begin{aussage}{\bf Any bilinear factorization system determines a strictly
associative pseudoalgebra of the 2-monad $\ad2$.} \label{as:fac>psalg}
This construction extends that in \cite{RosWoo}. If $\cE\hookrightarrow \cC
\hookleftarrow \cM$ is a bilinear factorization system, then applying the
notations in \ref{as:ds_fact}, for any morphism in $\cC^2$ as in the first
diagram in \eqref{eq:K^1_mor}, we can draw a fully commutative diagram in
$\cC$:
$$
\xymatrix @C=45pt{
a \ar[r]^-{i\epsilon(u)}\ar[d]_-{i\epsilon(f)}&
F(u)\ar[r]^-{i\mu(u)}\ar[d]^-{i\epsilon(i\epsilon(g)\,i\mu(u))}&
b\ar[d]^-{i\epsilon(g)}\\
F(f)\ar[r]^-{i\epsilon(i\epsilon(v)\,i\mu(f))}\ar[d]_-{i\mu(f)}&
F(vf)=F(gu)\ar[r]^-{i\mu(i\epsilon(g)\,i\mu(u))}
\ar[d]^-{i\mu(i\epsilon(v)\,i\mu(f))}&
F(g)\ar[d]^-{i\mu(g)}\\
x\ar[r]_-{i\epsilon(v)}&
F(v)\ar[r]_-{i\mu(v)}&
y}
$$
Indeed, the upper right square commutes since the path down-then-right
provides a factorization of the path right-then-down. By the same reasoning,
also the bottom left square is commutative. 
Combining the factorization property with the first equality in
\eqref{eq:e_bilin} and with the second equality in \eqref{eq:m_bilin},
respectively, we conclude that for any morphisms $f:a\to b$ and $g:b\to c$ in
$\cC$, 
\begin{equation}\label{eq:cond}
\epsilon(gf)=\epsilon(i\mu(g)\, i\epsilon(g)\, f)=
\epsilon(i\epsilon(g)\,f);\quad
\mu(gf)=\mu(g\,i\mu(f)\,i\epsilon(f))=\mu(g\, i\mu(f)).
\end{equation}
Applying this together with the second equality in \eqref{eq:e_bilin}, we see
that the right-then-down path in the top left square is equal to 
$$
i\epsilon\big(i\epsilon(g)\,i\mu(u)\big)i\epsilon(u)=
i\epsilon\big(g\, i\mu(u)\,i\epsilon(u)\big)=
i\epsilon(gu)
$$
while the down-then-right path is equal to 
$$
i\epsilon\big(i\epsilon(v)\,i\mu(f)\big)i\epsilon(f)=
i\epsilon\big(v\, i\mu(f)\,i\epsilon(f)\big)=
i\epsilon(vf).
$$
Hence also the top left square commutes and symmetrically so does the bottom
right square.  

With this information at hand, we define a functor $F:\cC^2\to \cC$ with the
morphism map
$$
\xymatrix{
a\ar[r]^-u\ar[d]_-f&b\ar[d]^-g\\x\ar[r]_-v&y}
\raise-.6cm\hbox{$\quad \mapsto \quad 
\xymatrix{
F(f)\ar[rr]^-{i\epsilon(i\epsilon(v)\,i\mu(f))}&&
F(vf)=F(gu)\ar[rr]^-{i\mu(i\epsilon(g)\,i\mu(u))}&&F(g).}$}
$$
In order to see that $F$ preserves identity morphisms, use \eqref{eq:cond};
\eqref{eq:e_bilin} and \eqref{eq:m_bilin}; and the factorization property to
deduce, for any morphism $f:a\to x$ in $\cC$, 
$$
F(1_f)=
i\mu(i\epsilon(f)i\mu(1_a))i\epsilon(i\epsilon(1_x)i\mu(f))=
i\mu i\epsilon(f) i\epsilon i\mu(f)=
i\mu(1_{F(f)})i\epsilon(1_{F(f)})=1_{F(f)}.
$$
For any commutative diagram 
$$
\xymatrix{
\ar[r]^-u\ar[d]_-f&\ar[r]^-w\ar[d]^-g&\ar[d]^-h\\\ar[r]_-v&\ar[r]_-t&}
$$
in $\cC$, it follows by \eqref{eq:cond} that
$$
F(wu,tv)=i\mu(i\epsilon(h)\,wu)\, i\epsilon(tv\, i\mu(f)).
$$
On the other hand, by applying \eqref{eq:cond} (in the first equality) and the
factorization property (in the second equality), we obtain 
\begin{eqnarray*}
F(w,t)F(u,v)
&=& i\mu(i\epsilon(h)\,i\mu(w))\, i\epsilon(t\, i\mu(g)) \,
i\mu(i\epsilon(g)\, u)\, i\epsilon(i\epsilon(v)\, i\mu(f))\\
&=&  i\mu(i\epsilon(h)\, \,i\mu(w))\
i\mu\big(i\epsilon(t\, i\mu(g))\, i\mu(i\epsilon(g)\, u) \big)\cdot \\
&&i\epsilon \big(i\epsilon(t\, i\mu(g))\, i\mu(i\epsilon(g)\, u) \big)\ 
i\epsilon(i\epsilon(v)\, i\mu(f)).
\end{eqnarray*}
In order to compare these expressions, note that by \eqref{eq:cond}; the
second condition in \eqref{eq:e_bilin} and the factorization property; and
since $tg=hw$,
$$
i\mu\big(i\epsilon(t\, i\mu(g))\, i\mu(i\epsilon(g)\, u) \big)=
i\mu\big(i\epsilon(t\, i\mu(g))\, i\epsilon(g)\, u \big)=
i\mu\big(i\epsilon(t\, g)\, u \big)=
i\mu\big(i\epsilon(h\, w)\, u \big).
$$
By the factorization property and the first condition in \eqref{eq:m_bilin};
by \eqref{eq:cond}; and by the second condition in 
\eqref{eq:e_bilin} and the factorization property, 
\begin{eqnarray*}
i\mu(i\epsilon(h)wu)
&=&i\mu(i\epsilon(h)\, \,i\mu(w))\
i\mu\big(i\epsilon(i\epsilon(h)\, i\mu(w))\, i\epsilon(w)\, u \big)\\
&=&i\mu(i\epsilon(h)\, \,i\mu(w))\
i\mu\big(i\epsilon(h\, i\mu(w))\, i\epsilon(w)\, u \big)\\
&=&i\mu(i\epsilon(h)\, \,i\mu(w))\ i\mu\big(i\epsilon(h\, w)\, u \big).
\end{eqnarray*}
This shows that $i\mu(i\epsilon(h)\, \,i\mu(w))\ 
i\mu\big(i\epsilon(t\, i\mu(g))\, i\mu(i\epsilon(g)\, u) \big)=
i\mu(i\epsilon(h)wu)$. 
By symmetrical considerations, $i\epsilon \big(i\epsilon(t\, i\mu(g))\,
i\mu(i\epsilon(g)\, u) \big)\ i\epsilon(i\epsilon(v)\, i\mu(f))=i\epsilon(tv\,
i\mu(f))$ -- proving that $F$ preserves composition as well.

In order to see strict associativity of $F$, consider the morphism in
$\cC^{2\x 2}$ on the left hand side of \eqref{eq:ad2_mp}. The functor $FM_\cC$
takes it to
$$
i\mu\big(i\epsilon(g'u')\, i\mu(k)\big)\ 
i\epsilon\big(i\epsilon(n)\, i\mu(gu)\big)
$$
while $FF^2$ takes it to 
$$
i\mu\big(i\epsilon(i\epsilon(v')\, i\mu(f'))\, 
i\mu(i\epsilon(f')\, i\mu(k))\big)  
i\epsilon\big(i\epsilon(i\epsilon(n)\, i\mu(g))\, 
i\mu(i\epsilon(g)\, i\mu(u))\big).
$$
These are equal morphisms $F(gu)\to F(g'u')$ in $\cC$ since by \eqref{eq:cond};
\eqref{eq:e_bilin} and the factorization property,
\begin{eqnarray*}
i\mu\big(i\epsilon(i\epsilon(v')\, i\mu(f'))\, 
i\mu(i\epsilon(f')\, i\mu(k))\big)  &=&
i\mu\big(i\epsilon(v'\, i\mu(f'))\, i\epsilon(f')\, k\big)\\
&=&
i\mu\big(i\epsilon(v'f')\, k \big) =
i\mu\big(i\epsilon(g'u')\, k \big),
\end{eqnarray*}
and symmetrically, $i\epsilon\big(i\epsilon(i\epsilon(n)\, i\mu(g))\, 
i\mu(i\epsilon(g)\, i\mu(u))\big)= i\epsilon\big(i\epsilon(n)\,
i\mu(gu)\big)$. 

The coherence natural isomorphism $\vartheta:FI_\cC\to 1_\cC$ is given by the
isomorphism $\vartheta_x=i\mu(1_x)$, with the inverse $\vartheta^{-1}_x=
i\epsilon(1_x)$, for any object $x$ of $\cC$. Indeed, 
$
\vartheta_x \vartheta_x^{-1} = i\mu(1_x) i\epsilon(1_x) =1_x
$ 
by the factorization property; and combining the factorization property with
\eqref{eq:cond}, \eqref{eq:e_bilin} and \eqref{eq:m_bilin}, also
\begin{eqnarray*}
\vartheta_x^{-1}\vartheta_x&=&
i\mu\big(i\epsilon(1_x) i\mu(1_x)\big) 
i\epsilon\big(i\epsilon(1_x) i\mu(1_x)\big)\\
&=&i\mu i\epsilon(1_x)i\epsilon i\mu(1_x)=
i\mu(1_{F(1_x)})i\epsilon(1_{F(1_x)})=
1_{F(1_x)}.
\end{eqnarray*}
Naturality of $\vartheta$ follows by the equality of
\begin{eqnarray*}
\vartheta_x FI_\cC(f)&=&
i\mu(1_x)i\mu\big(i\epsilon(1_x)i\mu(f)\big)
i\epsilon\big(i\epsilon(f)i\mu(1_a)\big) =
i\mu(f)i\epsilon(fi\mu(1_a))
\qquad\textrm{and}\\
f\vartheta_a&=&
fi\mu(1_a)=
i\mu\big(f\, i\mu(1_a)\big)
i\epsilon\big(f\, i\mu(1_a)\big)= 
i\mu(f)i\epsilon(fi\mu(1_a)),
\end{eqnarray*}
for any morphism $f:a\to x$ in $\cC$. Thus
$FI_\cC(f)=i\epsilon(1_x)fi\mu(1_a)$. Finally, $\vartheta$ obeys the coherence
conditions \eqref{eq:associ_psalg} since by the bilinearity conditions
\eqref{eq:e_bilin} and \eqref{eq:m_bilin}, $F(f)= F(i\mu(f)1_{F(f)}
i\epsilon(f))=F(1_{F(f)})$ and thus   
\begin{eqnarray*}
\vartheta_{F(f)}&=&
i\mu(1_{F(f)})=
i\mu i\mu(1_{F(f)}) i\epsilon i\mu(1_{F(f)})=
i\mu(1_{F(f)}) i\epsilon(1_{F(f)})=
1_{F(f)};\\
F(\vartheta^2_f)&=&
i\mu\big(i\epsilon(f)i\mu i\mu (1_a)\big)
i\epsilon\big(i\epsilon i\mu(1_x) i\mu(i\epsilon(1_x)f i\mu(1_a))\big)=
i\mu i\epsilon(f)\, i\epsilon i\mu(f) \\
&=&i\mu(1_{F(f)})\, i\epsilon(1_{F(f)})=1_{F(f)}.
\end{eqnarray*}
\end{aussage}

\begin{aussage}{\bf Any strictly associative pseudoalgebra of the 2-monad
    $\ad2$ determines a bilinear factorization system.} \label{as:psalg>fac} 
This construction, again, extends that in \cite{RosWoo}. Consider a strictly
associative pseudoalgebra $(\cC,F)$ of the 2-monad $\ad2$, with coherence
natural isomorphism $\vartheta:FI_\cC\to 1_\cC$. For any morphism $f:a\to x$
in $\cC$, the morphism $I_\cC (f)$ in $\cC^2$ has a factorization
$$
\raise-.6cm\hbox{$I_\cC (f)=$}
\xymatrix{a\ar[r]^-f\ar[d]_-{1_a}&x\ar[d]^-{1_x}\\a\ar[r]_-f&x}
 \raise-.6cm\hbox{$=$}
\xymatrix{a\ar[r]^-{1_a}\ar[d]_-{1_a}&a\ar[d]^-f\ar[r]^-f& x\ar[d]^-{1_x}\\
a\ar[r]_-f&x\ar[r]_-{1_x}&x}
$$
Applying $F$ to this equality, we can write $f$ in the equal form
\begin{equation}\label{eq:factor}
f=
\vartheta_x FI_\cC(f)\vartheta_a^{-1}=
\big(\xymatrix{
a\ar[r]^-{\vartheta_a^{-1}}&FI_\cC(a)\ar[r]^-{F(1_a,f)}&F(f)\ar[r]^-{F(f,1_x)}&
FI_\cC(x)\ar[r]^-{\vartheta_x}& x}\big).
\end{equation}
Using this decomposition of $f$, we construct two subcategories $\cE$ and
$\cM$ of $\cC$ as follows. Objects in both subcategories coincide with the
objects in $\cC$. Morphisms $a\to x$ in $\cE$ are those morphisms $e$ in $\cC$
for which $F(e)=FI_\cC(x)$ and $F(e,1_x)=1_{FI_\cC(x)}$. That is, morphisms of
the form 
\begin{equation}\label{eq:e_form} 
e=\big(
\xymatrix{
a\ar[r]^-{\vartheta_a^{-1}}&FI_\cC(a)\ar[r]^-{F(1_a,f)}&F(f)=FI_\cC(x)
\ar[r]^-{\vartheta_x}&x}
\big), 
\end{equation}
where $f:a\to x$ is any morphism in $\cC$ such that $F(f)=FI_\cC(x)$.
Indeed, for $e$ as in \eqref{eq:e_form}, 
$FI_\cC(e)=\big(\xymatrix{
F(1_a)\ar[r]^-{F(1_a,f)}&F(f)}\big)$. Hence $F(e,1_x)$ is equal to 
$$
\raise-.7cm\hbox{$F\left( \right .$}
\xymatrix @R=17pt @C=24pt{
a\ar[r]^-{\vartheta_a^{-1}}\ar[d]^(.3){e}&
F(1_a)\ar[r]^-{F(1_a,f)}\ar[d]^(.3){F(1_a,f)}&
F(1_x)\ar[r]^-{\vartheta_x}\ar[d]_(.7){1_{F(1_x)}}&
x\ar[d]_(.7){1_x}\\
x\ar[r]_-{\vartheta_x^{-1}}&
F(f)\ar@{=}[r]&
F(1_x)\ar[r]_-{\vartheta_x}&
x
}
\raise-.7cm\hbox{$\left . \right )= F\left ( \right .$}
\xymatrix @R=17pt @C=19pt{
F(1_a)\ar[r]^-{F(1_a,f)}\ar[d]^(.3){F(1_a,f)}&
F(f)\ar[d]_(.7){F(1_a,1_x)}\\
F(f)\ar[r]_-{F(1_a,1_x)}&
F(f)}
\raise-.7cm\hbox{$\left . \right )=F($}
\xymatrix{
a\ar[r]^-{1_a}\ar[d]^(.3){f}&
a\ar[d]_(.7){f}\\
x\ar[r]_-{1_x}&
x}
 \raise-.7cm\hbox{$\left . \right ),$}
$$
i.e. to $1_{F(f)}=1_{FI_\cC(x)}$. The first equality follows by the coherence
condition $F\vartheta^2=1_F$ and the second one follows by the associativity
of $F$ cf. \eqref{eq:F_associ}.
Symmetrically, the morphisms $a\to x$ in $\cM$ are those morphisms $m$ in
$\cC$ for which $F(m)=FI_\cC(a)$ and $F(1_a,m)=1_{FI_\cC(a)}$. That is, they
are of the form 
$$
\xymatrix{
a\ar[r]^-{\vartheta_a^{-1}}&FI_\cC(a)=F(f)\ar[r]^-{F(f,1_x)}&FI_\cC(x)
\ar[r]^-{\vartheta_x}&x},
$$
for morphisms $f:a\to x$ in $\cC$ such that $F(f)=FI_\cC(a)$.
The identity morphisms are clearly contained both in $\cE$ and $\cM$; we need
to check that they are closed under composition. For that the following lemma
will be needed.
\begin{lemma*} 
For any morphisms $f:y\to x$ in $\cC$ and $m:x\to z$ in $\cM$, the following
hold.
\begin{itemize}
\item[{(a)}] $F(mf)=F(f)$;
\item[{(b)}] $\big(\xymatrix{F(f)\ar[r]^-{F(1_y,m)}&F(mf)}\big)=1_{F(f)}$;
\item[{(c)}] $\big(\xymatrix{F(mf)\ar[r]^-{F(f,1_z)}&F(m)}\big)=
\big(\xymatrix{F(f)\ar[r]^-{F(f,1_x)}&F(1_x)}\big)$.
\end{itemize}
Symmetrically, for any morphisms $e:x\to y$ in $\cE$ and $f:y\to z$ in $\cC$,
the following hold.
\begin{itemize}
\item[{(a')}] $F(fe)=F(f)$;
\item[{(b')}] $\big(\xymatrix{F(fe)\ar[r]^-{F(e,1_z)}&F(f)}\big)=1_{F(f)}$;
\item[{(c')}] $\big(\xymatrix{F(e)\ar[r]^-{F(1_x,f)}&F(fe)}\big)=
\big(\xymatrix{F(1_y)\ar[r]^-{F(1_y,f)}&F(f)}\big)$.
\end{itemize}
\end{lemma*}

\begin{proof}
Evaluating the associativity constraint
  \eqref{eq:F_associ} on the morphism in $\cC^{2\x2}$ depicted in the first
  figure below, we obtain the equality in the second figure.
$$
\xymatrix @R=2pt @C=10pt{
&&&y\ar[rr]^-{f}\ar[dd]_-{1_y}&&x\ar[dd]^-{m}\\ 
&&&&&\\
&&&y\ar[rr]_(.3){mf}&&z\\
y\ar[rr]^-{f}\ar[dd]_-{1_y}\ar[rrruuu]^-{1_y}&&
x\ar[dd]^-{1_x}\ar[rrruuu]^(.6){1_x}&&&\\ 
&&&&&\\
y\ar[rr]_-{f}\ar[rrruuu]^(.7){1_y}|(.63){\color{white} IIII}&&
x\ar[rrruuu]_-{m}&&&}
\qquad
\raise-.7cm\hbox{$F\left ( \right .$}
\xymatrix{
F(1_y)\ar[r]^-{F(1_y,1_y)}\ar[d]_{F(f,f)}&F(1_y)\ar[d]^-{F(f,mf)}\\
F(1_x)\ar[r]_-{F(1_x,m)}&F(m)}
\raise-.7cm\hbox{$\left .\right) = F\left ( \right .$}
\xymatrix @R=30pt{
y\ar[d]_-f \ar[r]^-{1_y}&y\ar[d]^-{mf}\\
x\ar[r]_-m&z}
\raise-.7cm\hbox{$\left . \right )$}
$$
Since $m$ is a morphism in $\cM$, on the left hand side of this equality the
bottom arrow is the identity arrow $F(1_x,1_x)$ and both vertical arrows
are equal. Hence the left hand side is equal to $F(1_{F(f,f)})=1_{F(f)}$,
proving assertions (a) and (b).  

Applying $F^2$ to the morphism in $\cC^{2\x 2}$ in the first figure below, we
obtain the morphism in $\cC^2$ in the second figure.
$$
\xymatrix @R=2pt @C=10pt{
&&&y\ar[rr]^-{f}\ar[dd]_-{mf}&&x\ar[dd]^-{m}\\ 
&&&&&\\
&&&z\ar[rr]_(.3){1_z}&&z\\
y\ar[rr]^-{f}\ar[dd]_-{f}\ar[rrruuu]^-{1_y}&&
x\ar[dd]^-{1_x}\ar[rrruuu]^(.6){1_x}&&&\\ 
&&&&&\\
x\ar[rr]_-{1_x}\ar[rrruuu]^(.7){m}|(.63){\color{white} IIII}&&
x\ar[rrruuu]_-{m}&&&}
\qquad \qquad
\xymatrix{
F(f)\ar[r]^-{F(1_y,m)}\ar[d]_-{F(f,1_x)}&
F(mf)\ar[d]^-{F(f,1_z)}\\
F(1_x)\ar[r]_-{F(1_x,m)}&
F(m)}
$$
In the second figure, the top arrow is an identity morphism by part (b) and
the bottom arrow is an identity morphism since $m$ belongs to $\cM$. Thus
commutativity of the square implies part (c).

The remaining assertions follow by symmetrical reasonings.
\end{proof}
For any morphisms $m:y\to x$ and $n:x\to z$ in $\cM$, we need to show that
$$
\big(\xymatrix @=30pt{F(1_y)\ar[r]^-{F(1_y,nm)}&F(nm)}\big)=
\big(\xymatrix{F(1_y)\ar[r]^-{F(1_y,m)}&F(m)\ar[r]^-{F(1_y,n)}&F(nm)}\big)
$$
is an identity morphism. This holds because
the first arrow on the right hand side is an identity morphism since $m$
belongs to $\cM$ and the second arrow is an identity morphism by part (b) of
Lemma. This proves that $\cM$ is closed under composition and symmetrically,
so is $\cE$.

We have to construct a section for the 2-cell $\cE\cM\hookrightarrow
\cC\cC\stackrel\circ \to \cC$ in $\SM$. This is done by putting, for any
morphism $f:a\to x$ in $\cC$,
$$
\epsilon(f)
:=\big(
\xymatrix{a\ar[r]^-{\vartheta_a^{-1}}&F(1_a)\ar[r]^-{F(1_a,f)}&F(f)}\big);
\qquad 
\mu(f)
:=\big(
\xymatrix{F(f)\ar[r]^-{F(f,1_x)}&F(1_x)\ar[r]^-{\vartheta_x}&x}\big).
$$
In view of \eqref{eq:factor}, it gives a factorization $f=i\mu(f)i\epsilon(f)$ 
of $f$ (where $i$ stands for the obvious inclusions $\cE\hookrightarrow
\cC\hookleftarrow \cM$). 
It follows by the coherence of $\vartheta$ and associativity of $F$ that
$F(i\epsilon(f),1_{F(f)})$ is equal to  
$$
\raise-.7cm\hbox{$F\left ( \right .$}
\xymatrix @C=30pt {
a\ar[r]^-{\vartheta_a^{-1}}\ar[d]
_-{i\epsilon(f)}&
F(1_a)\ar[d]
^-{F(1_a,f)}\ar[r]^-{F(1_a,f)}&
F(f)\ar[d]
^-{F(1_a,1_x)}\\
F(f)\ar[r]_-{\vartheta^{-1}_{F(f)}=1_{F(f)}}&
F(f)\ar[r]_-{F(1_a,1_x)}&
F(f)}
\raise-.7cm\hbox{$\left . \right )=F\left( \right .$}
\xymatrix @=33pt {a\ar[d]
_-f\ar[r]^-{1_a}&a\ar[d]
^-f\\x\ar[r]_-{1_x}&x}
\raise-.7cm\hbox{$\left . \right ),$}
$$
i.e. to $1_{F(f)}$; and in particular $F(i\epsilon(f))=F(f)=F(1_{F(f)})$. 
Hence $\epsilon(f)$ belongs to $\cE$ and symmetrically, $\mu(f)$ belongs to
$\cM$. 
It remains to check properties \eqref{eq:e_bilin} and
\eqref{eq:m_bilin}. Consider any morphisms $e:x\to a$ in $\cE$, $f:a\to b$ in
$\cC$ and $m:b\to z$ in $\cM$. Then by part (a') in Lemma, $F(fe)=F(f)$ and by
part (b'), 
$$
\mu(fi(e))=
\big(\xymatrix{
F(fe)\ar[rr]^-{F(e,1_b)=1_{F(f)}}&&
F(f)\ar[r]^-{F(f,1_b)}&
F(1_b)\ar[r]^-{\vartheta_b}&b}
\big)=\mu(f).
$$
Moreover, by part (a), $F(mf)=F(f)$ and by part (c), $\mu(i(m)f)$ is equal to  
$$
\big(\xymatrix{
F(mf)\ar[r]^-{F(f,1_z)}&
F(m)\ar[rr]^-{FI_\cC(m)=F(m,1_z)}&&
F(1_z)\ar[r]^-{\vartheta_z}&z}
\big)=\big(
\xymatrix{
F(f)\ar[r]^-{F(f,1_b)}&
F(1_b)\ar[r]^-{\vartheta_b}&
b\ar[r]^-m&z}
\big),
$$
i.e. to $m\mu(f)$. This proves that $\mu$ obeys \eqref{eq:m_bilin} and
symmetrically, $\epsilon$ is shown to obey \eqref{eq:e_bilin}.
\end{aussage}

Our next aim is to relate the constructions in \ref{as:fac>psalg} and
\ref{as:psalg>fac}.

\begin{aussage}{\bf Strict morphisms of (strictly associative) pseudoalgebras
    of the 2-monad $\ad2$.} 
Recall from \cite{Kelly_lax_alg} that a strict morphism
$(\cC,F,\vartheta,\varphi)\to (\cC',F',\vartheta',\varphi')$ 
of pseudoalgebras for the 2-monad $\ad2$ is a functor $Q:\cC\to \cC'$ such
that the following diagrams (of functors and of natural transformations,
respectively) commute.
$$
\xymatrix @R=32pt{
\cC^2\ar[r]^-{Q^2}\ar[d]_-{F}&\cC^{\prime 2}\ar[d]^-{F'}\\\cC\ar[r]_-Q&\cC
}\qquad \qquad
\xymatrix @R=6pt{
QFI_\cC\ar[r]^-{Q\vartheta}\ar@{=}[d]&Q\ar@{=}[dd]\\
F'Q^2I_\cC\ar@{=}[d]&\\
F'I_{\cC'}Q\ar[r]_-{\vartheta' Q}&Q
}\qquad \qquad
\xymatrix @R=6pt{
QFM_\cC\ar[r]^-{Q\varphi}\ar@{=}[d]&QFF^2\ar@{=}[d]\\
F'Q^2M_\cC\ar@{=}[d]&F'Q^2F^2\ar@{=}[d]\\
F'M_{\cC'}Q^{2\x 2}\ar[r]_-{\varphi'Q^{2\x 2}}&F'F^{\prime 2}Q^{2\x 2}
}
$$
If the pseudoalgebras $(\cC,F,\vartheta,\varphi)$ and
$(\cC',F',\vartheta',\varphi')$ are strictly associative then 
the last diagram becomes trivial. 

(Strictly associative) pseudoalgebras of $\ad2$ and their strict morphisms
constitute a category. 
\end{aussage}

\begin{theorem} \label{thm:fac_eq_psalg}
There is an adjunction with a trivial counit, between the
category of bilinear factorization systems and the category of strictly
associative pseudoalgebras of the 2-monad $\ad2$ on $\Cat$. 
\end{theorem}

\begin{proof}
It is easy to see that both constructions in \ref{as:fac>psalg} and
\ref{as:psalg>fac} can be extended to functors $L$ and $R$ between the stated
categories, both acting on the morphisms as identity maps. 

Let $(\cC,F,\vartheta,1)$ be a strictly associative pseudoalgebra of $\ad2$
and consider the associated bilinear factorization system $\cE\hookrightarrow 
\cC\hookleftarrow \cM$ in \ref{as:psalg>fac}. For any morphism $(u,v):f\to g$
in $\cC^2$ (cf. first figure in \eqref{eq:K^1_mor}), 
$$
i\epsilon(v)i\mu(f)=
\big(\xymatrix{
F(f)\ar[r]^-{F(f,1_x)}&F(1_x)\ar[r]^-{F(1_x,v)}&F(v)}\big)=
\big(\xymatrix{
F(f)\ar[r]^-{F(f,v)}&F(v).}\big)
$$
Hence $i\epsilon\big(i\epsilon(v)i\mu(f)\big)$ is equal to
$$
\raise-.7cm\hbox{$F \left ( \right .$}
\xymatrix{
F(f)\ar[r]^-{F(1_a,1_x)}\ar[d]_-{F(1_a,1_x)}&F(f)\ar[d]^-{F(f,v)}\\
F(f)\ar[r]_-{F(f,v)}&F(v)}
\raise-.7cm\hbox{$\left . \right )=F \left ( \right .$}
\xymatrix{
a\ar[r]^-{1_a}\ar[d]_-{f}&a\ar[d]^-{vf}\\x\ar[r]_-v&y}
\raise-.7cm\hbox{$\left . \right )$}
$$
by the associativity of $F$.
Symmetrically, $i\mu\big(i\epsilon(g)i\mu(u)\big)= F(u,1_y):F(gu)\to F(g)$,
hence  
$$
\big(\xymatrix{
F(f)\ar[rr]^-{i\epsilon\big(i\epsilon(v)i\mu(f)\big)}&&F(vf)=F(gu)
\ar[rr]^-{i\mu\big(i\epsilon(g)i\mu(u)\big)}&&
F(g)}\big)=
\big(\xymatrix{
F(f)\ar[r]^-{F(u,v)}&F(g)}\big).
$$
This proves that the composite $LR$ is the identity functor.

Consider now a bilinear factorization system $\cE\stackrel i \hookrightarrow
\cC\stackrel i \hookleftarrow \cM$ with structure functions
$\epsilon:\mathrm{Mor}(\cC)\to \mathrm{Mor}(\cE)$ and $\mu:\mathrm{Mor}(\cC)
\to \mathrm{Mor}(\cM)$; and let $(\cC,F,\vartheta,1)$ be the associated
pseudoalgebra in \ref{as:fac>psalg}. For any morphism $f:a\to x$ in $\cC$, 
$$
F(1_a,f)\vartheta_a^{-1}
= i\mu\big(i\epsilon(f)\, i\mu(1_a)\big) \ 
i\epsilon\big(i\epsilon(f)\, i\mu(1_a)\big) \ i\epsilon(1_a)
=i\epsilon(f)\, i\mu(1_a)\, i\epsilon(1_a)
= i\epsilon(f).
$$
Symmetrically, $\vartheta_xF(f,1_x)=i\mu(f)$, hence $RL$ takes the
bilinear factorization system $\cE\hookrightarrow \cC\hookleftarrow \cM$
to 
$$
\{i\mu(1_x)\, i\epsilon(f)|a\stackrel f \to x; \ 
F(f)=F(1_x)\} 
\hookrightarrow \cC \hookleftarrow 
\{i\mu(f)\, i\epsilon(1_a)|a\stackrel f \to x; \ 
F(f)=F(1_a)\}.
$$
For a morphism $e:a\emor x$ in $\cE$, $i\mu(1_x)\, i\epsilon i(e)=
i\mu i(e)i\epsilon i(e)=i(e)$, and for a morphism $m:a\mmor x$ in $\cM$,
$i\mu i(m)\, i\epsilon(1_a)=i\mu i(m)i\epsilon i(m)=i(m)$. Thus a natural
transformation $\eta: 1\to RL$  
is given, at a bilinear factorization system $\cE\hookrightarrow
\cC\hookleftarrow \cM$, by the identity functor $1_\cC$. 
Evaluated at this object, $L\eta$ is the identity morphism of the
pseudoalgebra $LRL(\cE\hookrightarrow \cC\hookleftarrow
\cM)=L(\cE\hookrightarrow \cC\hookleftarrow \cM)$. Since the bilinear
factorization systems $R(\cC,F)$ and $RLR(\cC,F)$ are equal, for any strictly
associative pseudoalgebra $(\cC,F)$, also $\eta R$ is an identity natural
transformation.  
This proves that there is an adjunction $L \dashv R$, with trivial counit and
unit $\eta$. 
\end{proof}

\begin{aussage} {\bf Orthogonal factorization systems vs pseudoalgebras of 
$\ad2$.} \label{as:psalg_eq_ortfac}
Orthogonal factorization systems and pseudoalgebras of $\ad2$ were shown in
\cite{KorTho} to be equivalent notions. 
Indeed, for an orthogonal factorization system ${\overline \cE}\hookrightarrow
\cC\hookleftarrow {\overline \cM}$, a pseudoalgebra structure $F:\cC^2\to \cC$
is constructed using the diagonal fill-in property. A morphism in $\cC^2$ as
in the first diagram of \eqref{eq:K^1_mor} is taken by $F$ to the unique
diagonal arrow rendering commutative the following diagram, 
$$
\xymatrix@=30pt{
\ar[r]^-{e_g u}\ar[d]_-{e_f}&\ar[d]^-{m_g}\\ \ar@{-->}[ru]\ar[r]_-{vm_f}&}
$$
where $f=m_f e_f$ and $g=m_g f_g$ are factorizations in ${\overline \cE}$ and
${\overline \cM}$. 
Conversely, for a pseudoalgebra $(\cC,F)$ of $\ad2$, with coherence natural
isomorphism $\vartheta:FI_\cC\to 1_\cC$, there is an orthogonal factorization
system  
$$
{\overline \cE}:=\{f\in \cC\vert F(f,1_x)\ \textrm{is iso}\} 
\hookrightarrow \cC \hookleftarrow 
\{f\in \cC\vert F(1_a,f)\ \textrm{is iso}\}=:{\overline \cM},
$$
with factorization $f=\vartheta_x F(f,1_x)F(1_a,f)\vartheta^{-1}_a$ of any
morphism $f:a \to x$ in $\cC$.
\end{aussage}

\begin{corollary} \label{cor:fac>ort}
For any bilinear factorization system $\cE\stackrel i \hookrightarrow \cC
\stackrel i \hookleftarrow \cM$,  
there is a canonically associated orthogonal factorization system 
${\overline \cE}\hookrightarrow \cC\hookleftarrow {\overline \cM}$ such that
the functors $i$ factorize through ${\overline \cE}$ and ${\overline \cM}$,
respectively. 
\end{corollary}

\begin{proof}
Applying the functor $L$ in Theorem \ref{thm:fac_eq_psalg} to any bilinear
factorization system $\cE\hookrightarrow \cC\hookleftarrow \cM$, we obtain a
pseudoalgebra structure for $\ad2$ on $\cC$, with the action functor
$F:\cC^2\to \cC$ described in  \ref{as:fac>psalg}. By
\ref{as:psalg_eq_ortfac}, there is a corresponding orthogonal factorization
system 
$$
{\overline \cE}:=\{f \in \cC| i\mu(f)\ \textrm{is iso}\}
\hookrightarrow \cC\hookleftarrow 
\{f \in \cC| i\epsilon(f)\ \textrm{is iso}\}=:{\overline \cM},
$$ 
with factorization $f=i\mu(f) i\epsilon(f)$ of any
morphism $f$ in $\cC$. The category $\cE\cong i(\cE)$ is contained in
${\overline \cE}$  
since for any morphism $e:a\emor b$ in $\cE$, $i\mu i(e)=i\mu(1_b)$
(cf. \eqref{eq:m_bilin}) is an isomorphism with the inverse
$i\epsilon(1_b)$. Symmetrically, ${\overline \cM}\supseteq i(\cM)\cong \cM$.
\end{proof}

\begin{example}
The orthogonal factorization systems associated to the bilinear factorization
systems in Examples \ref{ex:order} and \ref{ex:groupoid} are trivial; that is,
in both cases ${\overline \cE}=\cC$ and the morphisms in ${\overline \cM}$ are
precisely the isomorphisms in $\cC$. (These orthogonal factorization systems
can be obtained also from strict ones in which $\cE=\cC$ and $\cM$ contains
precisely the identity morphisms in $\cC$.)

However, the orthogonal factorization system associated to the bilinear
factorization system in Example \ref{ex:non_triv} is non-trivial:
Unless $q$ (equivalently, $p$) is an isomorphism,
$p$ does not belong to ${\overline \cE}$ (whose non-identity morphisms
are $f$, $q$, $fq$ and $qp=f^{-1}$); and $q$ does not belong to ${\overline
\cM}$ (whose non-identity morphisms are $f$, $p$, $pf$ and
$qp=f^{-1}$). Thus in general none of the subcategories ${\overline \cE}$ and
${\overline \cM}$ is equal to $\cC$. 
\end{example}

\begin{corollary}
For any weak distributive law $\varrho:\cE\cM\to \cM\cE$ in $\SM$, there is an 
orthogonal factorization system ${\overline \cE}\hookrightarrow \cM_\varrho\cE
\hookleftarrow {\overline \cM}$ on the associated product category together
with identity-on-objects functors $\cE\to \overline{\cE}$ and $\cM\to
\overline{\cM}$.   
\end{corollary}

\begin{proof}
Applying the left adjoint functor in Corollary \ref{cor:wdl_fac} to any
weak distributive law $\varrho:\cE\cM\to \cM\cE$, we obtain a bilinear
factorization system $i({\cE})\hookrightarrow i(\cM)_{S(\varrho)} i(\cE)=
\cM_\varrho\cE \hookleftarrow i({\cM})$, where $S$ is the left adjoint functor
from Proposition \ref{prop:retract} and $i$ stands for the
identity-on-objects functors in Proposition \ref{prop:functor_i}. Hence the
claim follows by Corollary \ref{cor:fac>ort}. 
\end{proof}

\end{document}